\documentclass[12pt,reqno]{amsart}
\usepackage[notref,notcite]{}
\usepackage{amssymb,verbatim}
\usepackage{amsmath}
\usepackage[usenames]{color}

\newtheorem{theorem}{Theorem}[section]
\newtheorem{lemma}[theorem]{Lemma}
\newtheorem{proposition}[theorem]{Proposition}

\newtheorem{definition}[theorem]{Definition}

\newtheorem{conjecture}[theorem]{Conjecture}
\numberwithin{equation}{section}

\def\R{\mathbb R}
\def\S{\mathbb S}

\def\l{\langle}
\def\r{\rangle}

\newcommand{\I}{\textsl{i}}
\newcommand{\md}{\mathrm{d}}

\newcommand{\wh}{\widehat}

\newcommand{\vp}{\varphi}

\newcommand{\PD}{\partial}

\newcommand{\wt}{\widetilde}

\newcommand{\Sc}{\mathcal{S}}

\newcommand{\Rb}{\mathbb{R}}
\newcommand{\Sb}{\mathbb{S}}

\newcommand{\Beq}{\begin{equation}}
\newcommand{\Eeq}{\end{equation}}
\newcommand{\beq}{\begin{equation*}}
\newcommand{\eeq}{\end{equation*}}
\newcommand{\bal}{\begin{align}}
\newcommand{\eal}{\end{align}}

\newcommand{\n}{\nabla}

\newcommand{\bpr}{\begin{proof}}
	\newcommand{\epr}{\end{proof}}

\newcommand{\bel}[1]{\begin{equation}\label{#1}}
\newcommand{\ee}{\end{equation}}

\begin{document}
	
	\title[Reshetnyak formulas for the ray transform]{Ray transform on Sobolev spaces of symmetric tensor fields, I: Higher order Reshetnyak formulas}
	\author[V. P. Krishnan and V.\ A.\ Sharafutdinov]{Venkateswaran P.\ Krishnan$^\ast$ and Vladimir A.\ Sharafutdinov$^\sharp$}
	
	% Please provide minimum  5 keywords.
	
%	\thanks{The first author was supported by US NSF grant DMS 1616564.}
%	\thanks{The second author was supported by SERB National Postdoctoral fellowship, PDF/2017/002780.}
%	\thanks{First three authors were supported by  Airbus Corporate Foundation Chair grant
	%	``Mathematics of Complex Systems'' established at TIFR CAM and TIFR ICTS,
%		Bangalore, India.}

\thanks{The first author was supported by India SERB Matrics Grant MTR/2017/000837, and the second author was supported by RFBR, Grant 20-51-15004 (joint French -- Russian grant).}

	\address{
		\newline\indent$^\dagger$TIFR Centre for Applicable Mathematics, Sharada Nagar, Chikkabommasandra,\newline\indent\hspace{0mm} Yelahanka New Town, Bangalore, India
		\newline\indent$^\sharp$Sobolev Institute of Mathematics; 4 Koptyug Avenue, Novosibirsk, 630090, Russia.
         }

	\email{vkrishnan@tifrbng.res.in, sharaf@math.nsc.ru}

	%\keywords{Tensor tomography, ray transform.}
	
	%\subjclass[2000]{Primary 35R30; Secondary 35P99}
	
\begin{abstract}
For an integer $r\ge0$, we prove the $r$th order Reshetnyak formula for the ray transform of rank $m$ symmetric tensor fields on ${\R}^n$. Certain differential operators $A^{(m,r,l)}\ (0\le l\le r)$ on the sphere ${\S}^{n-1}$ are main ingredients of the formula. The operators are defined by an algorithm that can be applied for any $r$ although the volume of calculations grows fast with $r$. The algorithm is realized for small values of $r$ and Reshetnyak formulas of orders $0,1,2$ are presented in an explicit form.
\end{abstract}
	
	\maketitle
	\section{Introduction}
	The ray transform integrates functions or more generally symmetric tensor fields over lines in $\Rb^n$ and the Radon transform integrates functions over hyperplanes. The ray transform of functions is the main mathematical tool of computer tomography. The ray transform of vector fields and second rank tensor fields is used in Doppler tomography and travel time tomography. Note that the Radon transform and ray transform coincide up to parametrization in the 2-dimensional case.
	
	Three questions naturally arise in the study of the Radon transform or ray transform (or any other transform with tomographic applications) (i) inversion formulas (ii) stability and (iii) range characterization.
	We do not discuss inversion formulas referring the reader instead to the books \cite{HB} and \cite{mb}.
The main topic of the current paper is the stability question for the ray transform of symmetric tensor fields. Stability estimates are important in tomographic applications since data obtained by measurements always contain some errors. The range characterization will be the topic of our forthcoming paper.
	
	\section{Summary of prior results}
	
In this section, we define the Radon transform and the ray transform of symmetric tensor fields. Then we describe prior results on stability estimates for these transforms.

\subsection{Reshetnyak formulas for the Radon transform}
	The Radon transform $R$ integrates a function along hyperplanes. The set of hyperplanes can be parameterized by points of $\Sb^{n-1}\times \Rb$. Then $R$ is defined by
	\[
	Rf(\xi,p) = \int\limits_{\l\xi,x\r=p} f(x) \, dx\quad\big((\xi,p)\in\Sb^{n-1}\times \Rb\big),
	\]
where $\langle \cdot\: ,\cdot\rangle$ is the standard dot-product in $\Rb^n$ and $dx$ is the $(n-1)$-dimensional Lebesgue measure on the hyperplane $\{ x\mid \l\xi,x\r =p\}$.  Some condition on $f$ should be imposed for the integral above to converge.
	
	Let $\Sc(\Rb^n)$ be the Schwartz space of smooth functions rapidly decaying at infinity together with all derivatives (we use the term ``smooth'' as the synonym of ``$C^{\infty}$-smooth''). Similarly let $\Sc(\Sb^{n-1}\times \Rb)$ be the Schwartz space of functions $\varphi(\xi,p)$ on $\Sb^{n-1}\times \Rb$ rapidly decaying as $|p|\rightarrow\infty$ together with all derivatives. Both $\Sc(\Rb^n)$ and $\Sc(\Sb^{n-1}\times \Rb)$ are furnished with standard topologies.
	In fact, the  space ${\mathcal S}(E)$ is well defined for a smooth vector bundle $E\rightarrow M$ over a compact manifold $M$.
Let $\Sc_e(\Sb^{n-1}\times \Rb)$ be the closed subspace of $\Sc(\Sb^{n-1}\times \Rb)$ consisting of functions satisfying $\varphi(-\xi,-p)=\varphi(\xi,p)$.
	Then $R:\Sc(\Rb^n)\to \Sc_e(\Sb^{n-1}\times \Rb)$ is a bounded linear operator and it extends continuously to certain spaces of functions and distributions.
	
We use the Fourier transform $\Sc(\Rb^n) \to \Sc(\Rb^n)$, $f\mapsto \wh{f}$ in the form
\[
\wh{f}(y)=\frac{1}{(2\pi)^{n/2}} \int e^{-\I \l y, x\r} f(x) \, dx.
\]
The Fourier transform $\Sc(\Sb^{n-1}\times \Rb)\to\Sc(\Sb^{n-1}\times \Rb),\ \varphi(\xi,p)\mapsto\wh\varphi(\xi,q)$ is the one-dimensional Fourier transform in $p$ while $\xi\in \Sb^{n-1}$ is considered as a parameter:
\[
\wh{\vp}(\xi,q)=\frac{1}{2\pi} \int e^{-\I q p} \vp(\xi,p) \, dp.
\]
The well known {\it slice theorem} connects the Fourier transform of a function and the Fourier transform of the Radon transform of the function:
\[
\wh{Rf}(\xi,q)= \wh{f}(q\xi).
\]

For real $s$ and $t>-n/2$, the Hilbert space $H^s_t(\Rb^n)$ is defined as the completion of $\Sc(\Rb^n)$ with respect to the norm
\[
\lVert f\rVert_{H^s_t(\Rb^n)}^2 = \int\limits_{\Rb^n} |y|^{2t}(1+|y|^2)^{s-t}|\wh{f}(y)|^2 \, dy
\]
and for real $s$ and $t>-1/2$, the Hilbert space $H^s_{t,e}(\Sb^{n-1}\times \Rb)$ is defined as the completion of $\Sc_e(\Sb^{n-1}\times \Rb)$ with respect to the norm
\[
\lVert \vp\rVert^2_{H^s_t(\Sb^{n-1}\times \Rb)}=\frac{1}{2(2\pi)^{n-1}}\int\limits_{\Sb^{n-1}}\int\limits_{\Rb} |q|^{2t}(1+q^2)^{s-t}|\wh{\vp}(\xi,q)|^2\, dq \, d\xi,
\]
where $d\xi$ is the standard $(n-1)$-dimensional volume for on the sphere ${\S}^{n-1}$.
The zeroth order Reshetnyak formula for the Radon transform derived in \cite{Sh3} looks as follows:

\begin{theorem}\cite[Theorem 2.1]{Sh3} \label{Th2.1}
	For a real $s$ and $t>-n/2$, the equality
	\[
	\lVert f\rVert_{H^s_t(\Rb^n)} =\lVert Rf\rVert_{H^{s+(n-1)/2}_{t+(n-1)/2}(\Sb^{n-1}\times \Rb)}
	\]
	holds for all $f\in \Sc(\Rb^n)$. The Radon transform $R$ extends uniquely to an isometry
\[
R: H^s_t(\Rb^n)\to H^{s+(n-1)/2}_{t+(n-1)/2,e}(\Sb^{n-1}\times \Rb)
\]
\end{theorem}
The case $s=t=0$ was proved in an unpublished paper by Yu. Reshetnyak (circa 1960).

To write down higher order Reshetnyak formulas for the Radon transform, we next discuss a generalization of the Sobolev spaces introduced above. We reproduce some contents from \cite{Sh5} and then state the higher order Reshetnyak formula.

Let $\Delta_{\S}:C^\infty({\S}^{n-1})\rightarrow C^\infty({\S}^{n-1})$ be the spherical Laplacian (this operator was denoted by $\Delta_\xi$ in \cite{Sh5}, but now we reserve the notation $\Delta_\xi$ for another operator introduced below). We choose the sign of the Laplacian so that it is a non-negative operator.
For every real $r$, the Sobolev space $H^r({\S}^{n-1})$ can be defined as the completion of $C^\infty({\S}^{n-1})$ with respect to the norm:
\begin{equation}
\|\varphi\|^2_{H^r({\S}^{n-1})}=\|(\Delta_{\S}+{\mathbf 1})^{r/2}\varphi\|^2_{L^2({\S}^{n-1})}=
\int\limits_{{\S}^{n-1}}|(\Delta_{\S}+{\mathbf 1})^{r/2}\varphi(\xi)|^2\,d\xi,
	                                \label{2.1}
\end{equation}
where ${\mathbf 1}$ is the identity operator.
Spherical harmonics of degree $l$ are eigenfunctions of $\Delta_{\S}$
$$
\Delta_{\S} Y_{l}=\lambda(l,n) Y_{l},\quad \lambda(l,n)=l(l+n-2).
$$
Choosing an orthonormal basics $\{Y_{lm}\}_{m=1}^{N(n,l)}$ of the space of degree $l$ spherical harmonics,
a function $\varphi\in C^\infty({\S}^{n-1})$ is represented by the Fourier series
$$
\varphi(\xi)=\sum\limits_{l=0}^\infty\sum\limits_{m=1}^{N(n,l)}{\varphi}_{lm}Y_{lm}(\xi).
$$
Then the formula \eqref{2.1} can be equivalently written as
$$
\|\varphi\|^2_{H^r({\S}^{n-1})}=\sum\limits_{l=0}^\infty\big(\lambda(l,n)+1\big)^r\sum\limits_{m=1}^{N(n,l)}|{\varphi}_{lm}|^2.
$$

Given $f\in{\mathcal S}({\R}^n)$, we represent the Fourier transform $\wh f\in{\mathcal S}({\R}^n)$ by the series in spherical harmonics
\begin{equation}
	\wh f(y)=\sum\limits_{l=0}^\infty\sum\limits_{m=1}^{N(n,l)}{\wh f}_{lm}(|y|)Y_{lm}(y/|y|).
	                                \label{2.2}
\end{equation}
For arbitrary reals $r,s$ and for $t>-n/2$, we introduce the norm $\|\cdot\|_{H^{(r,s)}_t({\R}^n)}$ on
${\mathcal S}({\R}^n)$ by
\begin{equation}
	\|f\|^2_{H^{(r,s)}_t({\R}^n)}
	=\sum\limits_{l=0}^\infty\big(\lambda(l,n)+1\big)^r\sum\limits_{m=1}^{N(n,l)}
	\int\limits_0^\infty q^{2t+n-1}(1+q^2)^{s-t}\,|{\wh f}_{lm}(q)|^2\,dq,
																			\label{2.3}
\end{equation}
where ${\wh f}_{lm}$ are Fourier coefficients defined by \eqref{2.2}. Then we define the Hilbert space $H^{(r,s)}_t({\R}^n)$ as the completion of ${\mathcal S}({\R}^n)$ with respect to the norm \eqref{2.3}.

Similar arguments apply to functions $\varphi\in{\mathcal S}({\mathbb S}^{n-1}\times{\mathbb R})$. We represent the Fourier transform
$\wh\varphi\in{\mathcal S}({\mathbb S}^{n-1}\times{\mathbb R})$ by the series in spherical harmonics
\begin{equation}
	\wh \varphi(\xi,q)=\sum\limits_{l=0}^\infty\sum\limits_{m=1}^{N(n,l)}{\wh \varphi}_{lm}(q)Y_{lm}(\xi).
									\label{2.4}
\end{equation}
For reals $r,s$ and $t>-1/2$, we introduce the norm $\|\cdot\|_{H^{(r,s)}_t({\S}^{n-1}\times{\R})}$ on
${\mathcal S}({\S}^{n-1}\times{\R})$ by
\begin{equation}
	\|\varphi\|^2_{H^{(r,s)}_t({\S}^{n-1}\times{\R})}=\frac{1}{2(2\pi)^{n-1}}
	\sum\limits_{l=0}^\infty\big(\lambda(l,n)+1\big)^r\sum\limits_{m=1}^{N(n,l)}
	\int\limits_{-\infty}^\infty|q|^{2t}(1+q^2)^{s-t}\,|{\wh\varphi}_{lm}(q)|^2\,dq,
																			\label{2.5}
\end{equation}
and define the Hilbert space $H^{(r,s)}_{t,e}({\S}^{n-1}\times{\R})$ as the completion of
${\mathcal S}_e({\S}^{n-1}\times{\R})$ with respect to the norm \eqref{2.5}.
In the case of $r=0$, the spaces $H^{(0,s)}_t({\R}^n)$ and $H^{(0,s)}_{t,e}({\S}^{n-1}\times{\R})$ coincide with $H^s_t({\R}^n)$ and $H^s_{t,e}({\S}^{n-1}\times{\R})$ respectively.

\begin{theorem}\cite[Theorem 3.4]{Sh5} \label{Th2.2}
For arbitrary reals $r,s$ and $t>-n/2$,  the Radon transform
$R: \Sc(\Rb^n)\to \Sc(\Sb^{n-1}\times \Rb)$
uniquely extends to an isometry of the Hilbert spaces
\[
R:H^{(r,s)}_t(\Rb^n)\to H^{(r,s+(n-1)/2)}_{t+(n-1)/2,e}(\Sb^{n-1}\times \Rb).
\]
\end{theorem}

	\subsection{Reshetnyak formulas for the ray transform}
	Next we consider ray transform of symmetric tensor fields. Let $S^{m}\Rb^{n}$ be the  ${n+m-1}\choose m$-dimensional complex vector space of rank $m$ symmetric tensors on ${\R}^n$
	and
	$\Sc(\Rb^{n}; S^{m}\Rb^{n})$ be the Schwartz space of $S^{m}\Rb^{n}$-valued functions that are called rank $m$ smooth fast decaying symmetric tensor fields on ${\R}^n$.
	The family of oriented straight lines in ${\mathbb R}^n$ is parameterized by points of the manifold
	\[
	T{\mathbb S}^{n-1}=\{(x,\xi)\in{\mathbb R}^n\times{\mathbb R}^n\mid |\xi|=1,\langle x,\xi\rangle=0\}\subset{\mathbb R}^n\times{\mathbb R}^n,
	\]
	that is, the tangent bundle of the unit sphere ${\mathbb S}^{n-1}$. Namely, a point $(x,\xi)\in T{\mathbb S}^{n-1}$ determines the line $\{x+t\xi\mid t\in{\mathbb R}\}$. Similar to the case of the Radon transform, we consider the Schwartz space ${\mathcal S}(T{\mathbb S}^{n-1})$.
	The {\it ray transform} $I$ is initially a bounded linear operator
	\begin{equation}
		I: \Sc(\Rb^{n}; S^{m}\Rb^{n})\to \Sc (T\Sb^{n-1})
		                                 \label{2.6}
	\end{equation}
	that is defined, for $f=(f_{i_1\dots i_m})\in\Sc(\Rb^{n}; S^{m}\Rb^{n})$, by
	\begin{equation}
		If (x,\xi)=\int\limits_{-\infty}^\infty f_{i_1\dots i_m}(x+t\xi)\,\xi^{i_1}\dots\xi^{i_m}\,\md t=\int\limits_{-\infty}^\infty \l f(x+t\xi),\xi^m\r\,\md t\quad\big((x,\xi)\in T{\S}^{n-1}\big).
                                 		\label{2.7}
	\end{equation}
	We use the Einstein summation rule: the summation from 1 to $n$ is assumed over every index repeated in lower and upper positions in a monomial.
	To adopt our formulas to the summation rule, we use either lower or upper indices for denoting coordinates of vectors and tensors. For instance, $\xi^i=\xi_i$ in \eqref{2.7}. There is no difference between covariant and contravariant tensors since we use Cartesian coordinates only.
	Being initially defined by \eqref{2.7} on smooth fast decaying tensor fields, the operator \eqref{2.6} then extends to some wider spaces of tensor fields.
	
	Next we define certain Sobolev spaces similar to what we did for the Radon transform. The Fourier transform of a symmetric tensor field $f\in \Sc(\Rb^n; S^m\Rb^n)$ is defined component wise. The Fourier transform of a function $\vp(x,\xi)\in \Sc(T\Sb^{n-1})$ is defined as the $(n-1)$-dimensional Fourier transform over the subspace $\xi^{\perp}$:
	\[
	\wh{\vp}(y,\xi)=\frac{1}{(2\pi)^{(n-1)/2}}\int\limits_{\xi^{\perp}} e^{-\I \l y,x\r} g(x,\xi) \, d x\quad\big((y,\xi)\in T{\S}^{n-1}\big).
	\]
	The Sobolev space $H^s_t(\Rb^n;S^m\Rb^n)\ (t>-n/2)$ is defined as the completion of $\Sc(\Rb^n;S^m\Rb^n)$ with respect to the norm
	\[
	\lVert f\rVert^2=\int\limits_{\Rb^n} |y|^{2t}(1+|y|^2)^{s-t} |\wh{f}(y)|^2 \, dy,
	\]
	and for real $s$ and $t>-(n-1)/2$, the Sobolev space $H^{s}_{t}(T\Sb^{n-1})$ is defined as the completion of $\Sc(T\Sb^{n-1})$ with respect to the norm
	\Beq
	\lVert \vp\rVert^2_{H^s_t(T\Sb^{n-1})} = \frac{\Gamma\big(\frac{n-1}{2}\big)}{4\pi^{(n+1)/2}}\int\limits_{\Sb^{n-1}}\int\limits_{\xi^{\perp}} |y|^{2t}(1+|y|^2)^{s-t} |\wh{\vp}(y,\xi)|^2 \, dy \, d\xi.
											\label{2.8}
	\Eeq
	The zeroth order Reshetnyak formula for the ray transform of symmetric tensor fields was proved in \cite[Theorem 2.15.1]{mb} for $s=t=0$ and in \cite{Sh3} for arbitrary $(s,t)$. We do not present these theorems here since they are partial cases of a more general result of this paper; see Section 5.
	
Similar to what we did for the Radon transform, we would like to define more regular Sobolev spaces with the goal of deriving higher order Reshetnyak formulas. This is not at all straightforward. An important step in this direction was undertaken in \cite{Sh4} where a first order Reshetnyak formula for the ray transform of functions was proved.

We summarize the  existing results:% given in this section. We gave
\begin{enumerate}
	\item higher order Reshetnyak formulas for the Radon transform,
\item zeroth order Reshetnyak formula for the ray transform of tensor fields
\item first order Reshetnyak formula for the ray transform of \emph{functions}.
\end{enumerate}

An obvious question that arises is higher order Reshetnyak formulas for the ray transforms of functions, vector fields and symmetric tensor fields in general. This is the main goal of the current paper.

The outline of the paper is as follows. In Section \ref{S:Xi}, we define an operator $\Delta_\xi$ that serves as an analog of the spherical Laplacian. This is the most important operator necessary to define higher order versions of Sobolev norms on $T\Sb^{n-1}$. In Section 4, we consider so called tangential tensor fields well adopted to the foliation of ${\R}^n\setminus\{0\}$ into spheres centered at the origin. The Fourier transform of a solenoidal tensor field is a tangential tensor field. The main results of the paper, Theorems \ref{Th5.1} and \ref{Th5.2}, proved in Section \ref{S:MR}, give the Reshetnyak formula of an arbitrary integer order $r\ge0$ for rank $m$ symmetric tensor fields. The formula involves cirtain differential operators $A^{(m,r,l)}\ (0\le l\le r)$ on the sphere ${\S}^{n-1}$ which are defined by a long chain of formulas and recurrent relations. For $r=0,1,2$, we present explicit versions of the Reshetnyak formula in Section \ref{S:SC}.

\section{The spaces $H^{(r,s)}_t(TS^{n-1})$}   \label{S:Xi}

Our aim in the section is to define more regular, in terms of higher differentiability in the $\xi$-variable, Sobolev spaces on $T\Sb^{n-1}$.

We first recall some first order differential operators on $T\Sb^{n-1}$ introduced in \cite{KMSS}. Consider $\Rb^n\times \Rb^n$ with variables $(x,\xi)$  and introduce the following vector fields:
\begin{equation}
\begin{aligned}
\wt{X}_i&= \frac{\PD}{\PD x_i}-\xi_i \xi^p \frac{\PD}{\PD x^p},\\
\wt{\Xi}_i& = \frac{\PD}{\PD \xi_i} -x_i \xi^p \frac{\PD}{\PD x^p} - \xi_i\xi^p \frac{\PD}{\PD \xi^p}.
\end{aligned}
                                 		\label{3.1}
\end{equation}
As shown in \cite{KMSS}, these vector fields are tangent to $T\Sb^{n-1}$  at every point  $(x,\xi)\in T\Sb^{n-1}$, and therefore can be viewed as vector fields on $T\Sb^{n-1}$. Let $X_i$ and $\Xi_i$ be the restrictions of these vector fields on $T\Sb^{n-1}$.

We introduce the second order differential operator $\Delta_\xi$ on $T\Sb^{n-1}$ by
$$
\Delta_\xi=-\sum\limits_{i=1}^{n} \Xi_i^2.
$$
It is an invariant operator, i.e., independent of the choice of Cartesian coordinates.
This operator will be used for defining Sobolev spaces on $T\Sb^{n-1}$.

Recall from \eqref{2.8}, for $\varphi_j \in{\mathcal S}(T{\S}^{n-1})\ (j=1,2)$,
\begin{equation}
	(\varphi_1,\varphi_2)_{H^s_t(T{\S}^{n-1})}=\frac{\Gamma\big(\frac{n-1}{2}\big)}{4\pi^{(n+1)/2}}
	\int\limits_{{\S}^{n-1}}\int\limits_{\xi^\bot}|y|^{2t}(1+|y|^2)^{s-t}\wh{\varphi_1}(y,\xi)\,\overline{\wh{\varphi_2}(y,\xi)}\, d y d \xi.
	                                    \label{3.2}
\end{equation}

\begin{theorem} \label{Th3.1}
For every real $s$ and every $t>-(n-1)/2$, the adjoint of the operator $\Xi_i$ with respect to the $H^s_t(T{\S}^{n-1})$ inner product  \eqref{3.2} is expressed by
\begin{equation}
		\Xi_i^*=-\Xi_i+(n-1)\xi_i,
		                                \label{3.3}
\end{equation}
	where $\xi_i$ stands for the operator of multiplication by $\xi_i$.
\end{theorem}

\begin{proof}
	We start with the case of $s=t=0$.
	The $L^2$-product on $T{\S}^{n-1}$ is defined as
\begin{equation}
		(\varphi_1,\varphi_2)_{L^2(T{\S}^{n-1})}=\int\limits_{{\S}^{n-1}}\int\limits_{\xi^\bot}\varphi_1(x,\xi)\overline{\varphi_2(x,\xi)}\,dxd\xi.
	                                  	\label{3.4}
\end{equation}
	
	Given two functions $\varphi_j\in {\mathcal S}(T{\S}^{n-1})\ (j=1,2)$, define functions $\psi_j\in C^\infty\big({\R}^n\times({\R}^n\setminus\{0\})\big)$ by
$$
	\psi_j(x,\xi)=\varphi_j\Big(x-\frac{\l\xi,x\r}{\xi|^2}\xi,\frac{\xi}{|\xi|}\big).
$$
	These functions satisfy
\begin{equation}
		\psi_j(x+t\xi,\xi)=\psi_j(x,\xi)\ (t\in{\R}),\quad \psi_j(x,t\xi)=\psi_j(x,\xi)\ (0\neq t\in{\R}).
	                                    	\label{3.5}
\end{equation}
	Therefore
\begin{equation}
		\xi^p\frac{\partial\psi_j}{\partial x^p}=0,\quad \xi^p\frac{\partial\psi_j}{\partial \xi^p}=0.
	                                   	\label{3.6}
\end{equation}
	
	Since $\varphi_j=\psi_j|_{T{\S}^{n-1}}\ (j=1,2)$, we can write
\begin{equation}
		(\Xi_i\varphi_1,\varphi_2)_{L^2(T{\S}^{n-1})}=\int\limits_{{\S}^{n-1}}\int\limits_{\xi^\bot}
		(\tilde\Xi_i\psi_1)(x,\xi)\,\overline{\psi_2(x,\xi)}\,dxd\xi.
	                                   	\label{3.7}
\end{equation}
	
	By the definition of $\tilde\Xi_i$,
$$
	\tilde\Xi_i\psi_j=\frac{\partial\psi_j}{\partial \xi^i}-x_i\xi^p\frac{\partial\psi_j}{\partial x^p}-\xi_i\xi^p\frac{\partial\psi_j}{\partial \xi^p}.
$$
	Two last terms on the right-hand side are equal to zero by \eqref{3.6} and the formula simplifies to the following one:
\begin{equation}
		\tilde\Xi_i\psi_j=\frac{\partial\psi_j}{\partial \xi^i}\quad(j=1,2).
	                              	\label{3.8}
\end{equation}
	
	In view of \eqref{3.8}, formula \eqref{3.7} becomes
$$
	(\Xi_i\varphi_1,\varphi_2)_{L^2(T{\S}^{n-1})}=\int\limits_{{\S}^{n-1}}\int\limits_{\xi^\bot}
	\frac{\partial\psi_1}{\partial \xi^i}(x,\xi)\,\overline{\psi_2(x,\xi)}\,dxd\xi.
$$
	Quite similarly,
$$
	(\varphi_1,\Xi_i\varphi_2)_{L^2(T{\S}^{n-1})}=\int\limits_{{\S}^{n-1}}\int\limits_{\xi^\bot}
	\psi_1(x,\xi)\,\overline{\frac{\partial\psi_2}{\partial \xi^i}(x,\xi)}\,dxd\xi.
$$
	Taking the sum of two last equalities, we have
\begin{equation}
		(\Xi_i\varphi_1,\varphi_2)_{L^2(T{\S}^{n-1})}+(\varphi_1,\Xi_i\varphi_2)_{L^2(T{\S}^{n-1})}
		=\int\limits_{{\S}^{n-1}}\int\limits_{\xi^\bot}
		\frac{\partial}{\partial \xi^i}\big(\psi_1(x,\xi)\,\overline{\psi_2(x,\xi)}\big)\,dxd\xi.
		                      \label{3.9}
\end{equation}
	
	Define the function $g\in C^\infty({\R}^n\setminus\{0\})$ by
\begin{equation}
		g(\xi)=\int\limits_{\xi^\bot}\psi_1(x,\xi)\,\overline{\psi_2(x,\xi)}\,dxd\xi.
		                     \label{3.10}
\end{equation}
	Let us compute the derivative $\frac{\partial g}{\partial\xi^i}$. To this end we use the same trick as in the proof of \cite[Lemma 4.4]{KMSS}. Namely,
	fix a vector $\xi_0\in{\S}^{n-1}$. For an arbitrary vector $\xi\in{\R}^n\setminus\{0\}$ sufficiently close to $\xi_0$, the orthogonal projection
\begin{equation}
		\xi_0^\bot\rightarrow\xi^\bot,\quad x'\mapsto x=x'-\frac{\l\xi,x'\r}{|\xi|^2}\xi
		                         \label{3.11}
\end{equation}
	is one-to-one. We change the integration variable in \eqref{3.10} according to \eqref{3.11}.
	The Jacobian of the change is
$|\xi|\l\xi_0,\xi\r^{-1}$.
	After the change, formula \eqref{3.10} takes the form
$$
	g(\xi)=\frac{|\xi|}{\l\xi_0,\xi\r}\int\limits_{\xi_0^\bot}
	\psi_1\Big(x'-\frac{\l\xi,x'\r}{|\xi|^2}\xi,\xi\Big)\,\overline{\psi_2\Big(x'-\frac{\l\xi,x'\r}{|\xi|^2}\xi,\xi\Big)}\, d x'.
$$
	With the help of \eqref{3.5}, this formula is simplified to the following one:
$$
	g(\xi)=\frac{|\xi|}{\l\xi_0,\xi\r}\int\limits_{\xi_0^\bot}
	\psi_1(x',\xi)\,\overline{\psi_2(x',\xi)}\, d x'.
$$
	We can now differentiate this equality with respect to
	$\xi^i$
$$
\begin{aligned}
		\frac{\partial g}{\partial\xi^i}(\xi)
		&=\frac{\l\xi_0,\xi\r\xi_i-|\xi|^2\xi_{0,i}}{|\xi|\l\xi_0,\xi\r^2}\int\limits_{\xi_0^\bot}
		\psi_1(x',\xi)\,\overline{\psi_2(x',\xi)}\, d x'\\
		&+\frac{|\xi|}{\l\xi_0,\xi\r}\int\limits_{\xi_0^\bot}
		\frac{\partial\psi_1}{\partial\xi^i}\,\overline{\psi_2(x',\xi)}\, d x'
		+\frac{|\xi|}{\l\xi_0,\xi\r}\int\limits_{\xi_0^\bot}
		\psi_1(x',\xi)\,\overline{\frac{\partial\psi_2}{\partial\xi^i}(x',\xi)}\, d x'.
\end{aligned}
$$
	
	On assuming
	$\xi_0\in {\S}^{n-1}$,
	we set
	$\xi=\xi_0$
	in the latter formula. The formula simplifies to the following one:
$$
	\frac{\partial g}{\partial\xi^i}(\xi_0)=
	\int\limits_{\xi_0^\bot}\frac{\partial\psi_1}{\partial\xi^i}\,\overline{\psi_2(x',\xi)}\, d x'
	+\int\limits_{\xi_0^\bot}\psi_1(x',\xi)\,\overline{\frac{\partial\psi_2}{\partial\xi^i}(x',\xi)}\, d x'.
$$
	Replacing the notations
	$\xi_0$
	and
	$x'$
	with
	$\xi$
	and
	$x$
	respectively, we obtain
\begin{equation}
		\frac{\partial g}{\partial\xi^i}(\xi)=
		\int\limits_{\xi^\bot}\frac{\partial}{\partial\xi^i}
		\Big(\psi_1(x,\xi)\,\overline{\psi_2(x,\xi)}\Big)(x,\xi)\, d x
		\quad\mbox{for}\quad\xi\in {\S}^{n-1}.
		                      \label{3.12}
\end{equation}
	
	We use the following obvious fact. If a function $f\in C({\R}^n\setminus\{0\})$ is positively homogeneous of degree $\lambda>-n$, then
$$
	\int\limits_{{\S}^{n-1}}f(\xi)\,d\xi=(\lambda+n)\int\limits_{|z|\le1}f(z)\,dz.
$$
	We apply this fact to the function $\frac{\partial g}{\partial\xi^i}$ that is positively homogeneous of degree $-1$ as is seen from \eqref{3.5} and \eqref{3.12}. Thus,
$$
	\int\limits_{{\S}^{n-1}}\frac{\partial g}{\partial\xi^i}(\xi)\,d\xi=(n-1)\int\limits_{|z|\le1}\frac{\partial g(z)}{\partial z^i}\,dz.
$$
	Transforming the right-hand integral with the help of the divergence theorem, we obtain
$$
	\int\limits_{{\S}^{n-1}}\frac{\partial g}{\partial\xi^i}(\xi)\,d\xi=(n-1)\int\limits_{{\S}^{n-1}}\xi_i\,g(\xi)\,d\xi.
$$
	Together with \eqref{3.10} and \eqref{3.12}, this gives
$$
	\int\limits_{{\S}^{n-1}}\int\limits_{\xi^\bot}\frac{\partial}{\partial\xi^i}
	\Big(\psi_1(x,\xi)\,\overline{\psi_2(x,\xi)}\Big)(x,\xi)\, d xd\xi
	=(n-1)\int\limits_{{\S}^{n-1}}\int\limits_{\xi^\bot}\xi_i\,\varphi_1(x,\xi)\,\overline{\varphi_2(x,\xi)}\,dxd\xi.
$$
	
	With the help of the last formula, the equality \eqref{3.9} takes the form
$$
	(\Xi_i\varphi_1,\varphi_2)_{L^2(T{\S}^{n-1})}+(\varphi_1,\Xi_i\varphi_2)_{L^2(T{\S}^{n-1})}
	=(n-1)(\varphi_1,\xi_i\varphi_2)_{L^2(T{\S}^{n-1})}.
$$
	This is equivalent to \eqref{3.3} in the case of $s=t=0$.
	
	Let us now prove \eqref{3.3} for an arbitrary $s\in\R$ and $t>-(n-1)/2$. In view of \eqref{3.4}, the definition \eqref{3.2} can be written as
$$
	(\varphi_1,\varphi_2)_{H^s_t(T{\S}^{n-1})}=(\wh\varphi_1,w\wh\varphi_2)_{L^2(T{\S}^{n-1})},
$$
	where the weight $w$ is defined by
$$
	w=w(|y|)=\frac{\Gamma\big(\frac{n-1}{2}\big)}{4\pi^{(n+1)/2}}\,|y|^{2t}(1+|y|^2)^{s-t}.
$$
	Observe that $\Xi_iw=0$. Indeed, $\Xi_i|y|^2=0$ as immediately follows from the definition of $\Xi_i$.
	
	First of all,
$$
	(\Xi_i\varphi_1,\varphi_2)_{H^s_t(T{\S}^{n-1})}=(\wh{\Xi_i\varphi_1},w\wh\varphi_2)_{L^2(T{\S}^{n-1})}.
$$
	Since $\wh{\Xi_i\varphi_1}=\Xi_i\wh{\varphi_1}$ by \cite[Lemma 4.4]{KMSS}, this can be written as
$$
	(\Xi_i\varphi_1,\varphi_2)_{H^s_t(T{\S}^{n-1})}=(\Xi_i\wh{\varphi_1},w\wh\varphi_2)_{L^2(T{\S}^{n-1})}.
$$
	Applying \eqref{3.3} for $s=t=0$, we obtain
	$$
	(\Xi_i\varphi_1,\varphi_2)_{H^s_t(T{\S}^{n-1})}=\Big(\wh{\varphi_1},\big(-\Xi_i+(n-1)\xi_i\big)(w\wh\varphi_2)\Big)_{L^2(T{\S}^{n-1})}.
$$
	Since $\Xi_iw=0$, this can be written as
$$
	(\Xi_i\varphi_1,\varphi_2)_{H^s_t(T{\S}^{n-1})}=\Big(\wh{\varphi_1},w\big(-\Xi_i+(n-1)\xi_i\big)\wh\varphi_2\Big)_{L^2(T{\S}^{n-1})}.
$$
	Transforming the right-hand side in the reverse order, we see that
$$
\begin{aligned}
		 (\Xi_i\varphi_1,\varphi_2)_{H^s_t(T{\S}^{n-1})}&=\Big(\wh{\varphi_1},w\big(-\wh{\Xi_i\varphi_2}+(n-1)\wh{\xi_i\varphi_2}\big)\Big)_{L^2(T{\S}^{n-1})}\\
		&=\Big(\varphi_1,\big(-\Xi_i+(n-1)\xi_i\big)\varphi_2\Big)_{H^s_t(T{\S}^{n-1})}.
\end{aligned}
$$
\end{proof}

Theorem \ref{Th3.1} has the following corollary.

\begin{lemma} \label{L3.1}
	The operator
\begin{equation}
		\Delta_\xi=-\sum\limits_{i=1}^n\Xi_i^2=-\Xi^i\Xi_i:{\mathcal S}(T{\S}^{n-1})\rightarrow{\mathcal S}(T{\S}^{n-1})
		                          \label{3.13}
\end{equation}
	is non-negative with respect to the $H^s_t(T{\S}^{n-1})$-product for any real $s$ and $t>-(n-1)/2$.
\end{lemma}

\begin{proof}
	For $\varphi\in{\mathcal S}(T{\S}^{n-1})$ by Theorem \ref{Th3.1},
$$
\begin{aligned}
		\big(\Delta_\xi\varphi,\varphi\big)_{H^s_t(T{\S}^{n-1})}
		&=-\big(\Xi_i\Xi^i\varphi,\varphi\big)_{H^s_t(T{\S}^{n-1})}
		=-\big(\Xi^i\varphi,\Xi_i^*\varphi\big)_{H^s_t(T{\S}^{n-1})}\\
		&=\big(\Xi^i\varphi,\Xi_i\varphi\big)_{H^s_t(T{\S}^{n-1})}-(n-1)\big(\Xi^i\varphi,\xi_i\varphi\big)_{H^s_t(T{\S}^{n-1})}\\
		&=\big(\Xi^i\varphi,\Xi_i\varphi\big)_{H^s_t(T{\S}^{n-1})}-(n-1)\big(\xi_i\Xi^i\varphi,\varphi\big)_{H^s_t(T{\S}^{n-1})}.
\end{aligned}
$$
	The last term on the right-hand side is equal to zero since $\xi_i\Xi^i=0$. Thus,
\begin{equation}
		\big(\Delta_\xi\varphi,\varphi\big)_{H^s_t(T{\S}^{n-1})}
		=\sum\limits_{i=1}^n\|\Xi_i\varphi\|^2_{H^s_t(T{\S}^{n-1})}\ge0.
		                       \label{3.14}
\end{equation}
\end{proof}

We will widely use the operators
$$
\l\xi,\partial_\xi\r=\xi^p\frac{\partial}{\partial\xi^p},\quad
\l\xi,\partial_x\r=\xi^p\frac{\partial}{\partial x^p},\quad
\l x,\partial_\xi\r=x^p\frac{\partial}{\partial\xi^p},\quad
\l x,\partial_x\r=x^p\frac{\partial}{\partial x^p}.
$$

\begin{lemma} \label{L3.2}
Given $\varphi\in{\mathcal S}(T{\S}^{n-1})$, let a function
$\psi\in C^\infty\big({\R}^n\times({\R}^n\setminus\{0\})\big)$ satisfy $\psi|_{T{\S}^{n-1}}=\varphi$. Then
\begin{equation}
	\Delta_\xi\varphi=\Big[\Big(-\sum\limits_{i=1}^n\frac{\partial^2}{\partial\xi_i^2}
	+\l\xi,\partial_\xi\r^2+(n\!-\!2)\l\xi,\partial_\xi\r
	-|x|^2\l\xi,\partial_x\r^2
	+2\l x,\partial_\xi\r\l\xi,\partial_x\r
	-\l x,\partial_x\r\Big)\psi\Big]_{T{\S}^{n-1}}.
	                                      \label{3.15}
\end{equation}
\end{lemma}

\begin{proof}
$\Xi_i^2\varphi=\tilde\Xi_i^2\psi|_{T{\S}^{n-1}}$, where
$$
\tilde\Xi_i^2\psi=\Big(\frac{\partial}{\partial\xi^i}-x_i\xi^p\frac{\partial}{\partial x^p}-\xi_i\xi^p\frac{\partial}{\partial\xi^p}\Big)
\Big(\frac{\partial}{\partial\xi^i}-x_i\xi^q\frac{\partial}{\partial x^q}-\xi_i\xi^q\frac{\partial}{\partial\xi^q}\Big)\psi.
$$
After opening parentheses,
$$
\begin{aligned}
	\tilde\Xi_i^2\varphi&=\Big(\frac{\partial^2}{\partial\xi_i^2}
	+\xi_i^2\xi^p\xi^q\frac{\partial^2}{\partial\xi^p\partial\xi^q}
	-2\xi_i\xi^p\frac{\partial^2}{\partial\xi^i\partial\xi^p}\\
	&-2x_i\xi^p\frac{\partial^2}{\partial x^p\partial\xi^i}
	+x_i^2\xi^p\xi^q\frac{\partial^2}{\partial x^p\partial x^q}
	+2x_i\xi_i\xi^p\xi^q\frac{\partial^2}{\partial x^p\partial x^q}\\
	&-x_i\frac{\partial}{\partial x^i}-\xi^p\frac{\partial}{\partial\xi^p}-\xi_i\frac{\partial}{\partial\xi^i}
	+2\xi_i^2\xi^p\frac{\partial}{\partial\xi^p}+2x_i\xi_i\xi^p\frac{\partial}{\partial x^p}\Big)\psi.
\end{aligned}
$$
We restrict this equality to $T{\S}^{n-1}$, where $|\xi|=1$ and $\l x,\xi\r=0$. Performing the summation over $i$, we obtain
\begin{equation}
\begin{aligned}
		\Delta_\xi\varphi=\bigg[&\Big(-\sum\limits_{i=1}^n\frac{\partial^2}{\partial\xi_i^2}
		+\xi^p\xi^q\frac{\partial^2}{\partial\xi^p\partial\xi^q}
		-|x|^2\xi^p\xi^q\frac{\partial^2}{\partial x^p\partial x^q}
		+2x^q\xi^p\frac{\partial^2}{\partial x^p\partial\xi^q}\\
		&+x^p\frac{\partial}{\partial x^p}+(n-1)\xi^p\frac{\partial}{\partial \xi^p}\Big)\psi\Big]_{T{\S}^{n-1}}.
\end{aligned}
	                                 \label{3.16}
\end{equation}

Obviously,
$$
\begin{aligned}
	\xi^p\xi^q\frac{\partial^2}{\partial\xi^p\partial\xi^q}&=\l\xi,\partial_\xi\r^2-\l\xi,\partial_\xi\r,\\
	\xi^p\xi^q\frac{\partial^2}{\partial x^p\partial x^q}&=\l\xi,\partial_x\r^2,\\
	x^q\xi^p\frac{\partial^2}{\partial x^p\partial\xi^q}&=\l x,\partial_\xi\r\l\xi,\partial_x\r-\l x,\partial_x\r.
\end{aligned}
$$
Substituting these values into \eqref{3.16}, we obtain \eqref{3.15}.
\end{proof}

In what follows, we will mostly use the following partial case of Lemma \ref{L3.2}.
Given a function $\varphi\in{\mathcal S}(T{\S}^{n-1})$ and an integer $m\ge0$, define the function
$\psi\in C^\infty\big({\R}^n\times({\R}^n\setminus\{0\})\big)$ by
\begin{equation}
	\psi(x,\xi)=|\xi|^m\varphi\Big(x-\frac{\l x,\xi\r}{|\xi|^2}\xi,\frac{\xi}{|\xi|}\Big).
	                                \label{3.17}
\end{equation}
Then $\psi|_{T{\S}^{n-1}}=\varphi$ and the formula \eqref{3.15} is valid. The function $\psi$ satisfies
\begin{equation}
	\psi(x+t\xi,\xi)=\psi(x,\xi)\ (t\in\R),\quad \psi(x,t\xi)=t^m\psi(x,\xi)\ (0\neq t\in\R).
	                               \label{3.18}
\end{equation}
This implies
\begin{equation}
	\l\xi,\partial_x\r\psi=0,\quad \l\xi,\partial_\xi\r\psi=m\psi.
	                                   \label{3.19}
\end{equation}
The formula \eqref{3.15} is now simplified to the following one:
\begin{equation}
	\Delta_\xi\varphi=\Big[\Big(-\sum\limits_{i=1}^n\frac{\partial^2}{\partial\xi_i^2}
	-\l x,\partial_x\r+m(m+n-2)\Big)\psi\Big]_{T{\S}^{n-1}}.
	                                   \label{3.20}
\end{equation}

\begin{definition} \label{D3.1}
For an integer $r\ge0$, real $s$ and $t>-(n-1)/2$, introduce the norm on ${\mathcal S}(T{\S}^{n-1})$ (${\mathbf 1}$ is the identity operator)
\begin{equation}
		\|\varphi\|^2_{H^{(r,s)}_{t}(T{\S}^{n-1})}=\big(({\mathbf 1}+\Delta_\xi)^r\varphi,\varphi\big)_{H^s_t(T{\S}^{n-1})}
=\sum\limits_{l=0}^r{r\choose l}\big(\Delta_\xi^l\varphi,\varphi\big)_{H^s_t(T{\S}^{n-1})}
                                        		\label{3.21}
\end{equation}
and define the Hilbert space $H^{(r,s)}_t(T{\S}^{n-1})$ as the completion of ${\mathcal S}(T{\S}^{n-1})$ with respect to the norm \eqref{3.21}.
\end{definition}

Comparing \eqref{3.21} with \eqref{2.1} and \eqref{2.5}, we see the analogy: the operator $\Delta_\xi$ is used in the Definition \ref{D3.1} in the same way as the spherical Laplacian $\Delta_{\S}$ is used in the definition of $H^{(r,s)}_t({\S}^{n-1}\times\R)$.
However, there is the important difference between these operators: $\Delta_\xi$ is not an elliptic operator. Therefore we cannot use powers $(\Delta_\xi+E)^r$ with arbitrary real $r$. This is the main reason why the spaces $H^{(r,s)}_t(T{\S}^{n-1})$ are defined for integer $r\ge0$ only.
Unlike ${\S}^{n-1}\times\R$, the tangent bundle $T{\S}^{n-1}\rightarrow {\S}^{n-1}$ is not trivial (with exceptions of $n=2,4,8$); therefore the usage of spherical harmonics on $T{\S}^{n-1}$ is rather problematic.

\section{Tangential tensor fields}

Unlike the Radon transform, the operator \eqref{2.6} is not injective in the case of $m>0$. Given $If$ for a tensor field
$f\in{\mathcal S}({\R}^n;S^m{\R}^n)$, we can recover the solenoidal part of $f$ only, see \cite[Section 2.12]{mb}. Therefore the Reshetnyak formulas make sense on the space of solenoidal symmetric tensor fields.

We distinguish the subspace ${\mathcal S}_{sol}({\R}^n;S^m{\R}^n)$ in ${\mathcal S}({\R}^n;S^m{\R}^n)$ consisting of tensor fields
$g\in{\mathcal S}({\R}^n;S^m{\R}^n)$ satisfying
\begin{equation}
\sum\limits_{p=1}^n\frac{\partial g_{pi_2\dots i_m}}{\partial x^p}=0.
	                         \label{4.1}
\end{equation}
Such $g$ are called (smooth fast decaying) {\it solenoidal tensor fields} of rank $m$. The equation \eqref{4.1} is equivalently written in terms of the Fourier transform $f=\wh g$ as
\begin{equation}
y^pf_{pi_2\dots i_m}(y)=0.
	                         \label{4.2}
\end{equation}
Recall that we use $y$ as the Fourier dual variable of $x$.
In other words, the Fourier transform maps ${\mathcal S}_{sol}({\R}^n;S^m{\R}^n)$ isomorphically onto the subspace
${\mathcal S}_\top({\R}^n;S^m{\R}^n)\subset{\mathcal S}({\R}^n;S^m{\R}^n)$ consisting of tensor fields $f$ satisfying \eqref{4.2}. Such $f$ will be called
(smooth fast decaying) {\it tangential tensor fields} of rank $m$. Here the term ``tangential'' is used in the sense ``tangent to spheres centered at the origin''.

We are going to prove the Reshetnyak formula
\begin{equation}
\|g\|_{H^{(r,s)}_{t,sol}({\R}^n;S^m{\R}^n)}=\|Ig\|_{H^{(r,s)}_t(T{\S}^{n-1})}
	                         \label{4.3}
\end{equation}
for a solenoidal tensor field $g\in{\mathcal S}_{sol}({\R}^n;S^m{\R}^n)$. The norm on the right-hand side of \eqref{4.3} was defined in the previous section, see \eqref{3.21}. But the norm on the left-hand side of \eqref{4.3} is not defined yet. The latter norm will appear in the process of the proof.
The Reshetnyak formula \eqref{4.3} will be proved in the next section. In the current section, we will develop some machinery for treating tangential tensor fields.

Let us recall the main relation between the ray transform and Fourier transform \cite[formula 2.1.15]{mb}. If $f=\wh g$ for a tensor field $g\in{\mathcal S}({\R}^n;S^m{\R}^n)$, then
\begin{equation}
\wh{Ig}(y,\xi)=(2\pi)^{1/2}f_{i_1\dots i_m}(y)\xi^{i_1}\dots\xi^{i_m}\quad\mbox{for}\quad(y,\xi)\in T{\S}^{n-1}.
	                         \label{4.4}
\end{equation}

We have thus to compute the norm $\|\varphi\|_{H^{(r,s)}_t(T{\S}^{n-1})}$ of the function
\begin{equation}
	\varphi(y,\xi)=\big(f_{i_1\dots i_m}(y)\xi^{i_1}\dots\xi^{i_m}\big)\big|_{T{\S}^{n-1}}
	                   \label{4.5}
\end{equation}
for a tangential tensor field $f\in{\mathcal S}_\top({\R}^n;S^m{\R}^n)$. We start with computing $\Delta_\xi\varphi$. To this end we use formulas of the previous section with the following modification. In view of \eqref{4.2} and \eqref{4.5}, points of the manifold $T{\S}^{n-1}$ are denoted by $(y,\xi)$ in this section, since they are actually treated as Fourier dual variables of $(x,\xi)$.

For a function $\varphi$ defined by \eqref{4.5}, the formula \eqref{3.17} simplifies a little bit:
\begin{equation}
	\psi(y,\xi)=f_{i_1\dots i_m}\big(y-\frac{\l y,\xi\r}{|\xi|^2}\xi\big)\xi^{i_1}\dots\xi^{i_m}.
                               	\label{4.6}
\end{equation}
We differentiate this equality with respect to $\xi^i$
$$
\begin{aligned}
	\frac{\partial\psi}{\partial\xi^i}(y,\xi)&=mf_{ii_2\dots i_m}\Big(y-\frac{\l y,\xi\r}{|\xi|^2}\xi\Big)\xi^{i_2}\dots\xi^{i_m}\\
	&-\Big(\frac{\l y,\xi\r}{|\xi|^2}\delta^j_i+\frac{y_i\xi^j}{|\xi|^2}-2\frac{\l y,\xi\r\xi_i\xi^j}{|\xi|^4}\Big)
	\frac{\partial f_{i_1\dots i_m}}{\partial y^j}\Big(y-\frac{\l y,\xi\r}{|\xi|^2}\xi\Big)\xi^{i_1}\dots\xi^{i_m},
\end{aligned}
$$
where $\delta^j_i=\delta_{ij}=\delta^{ij}$ is the Kronecker tensor.
We again differentiate this equality with respect to $\xi^i$ and then set $|\xi|=1$ and $\l y,\xi\r=0$ in the resulting formula. In this way we obtain
$$
\begin{aligned}
	\frac{\partial^2\psi}{\partial\xi_i^2}\Big|_{T{\S}^{n-1}}&=m(m-1)f_{iii_3\dots i_m}(y)\xi^{i_3}\dots\xi^{i_m}
	-2m\,y_i\xi^j\,\frac{\partial f_{ii_2\dots i_m}}{\partial y^j}(y)\xi^{i_2}\dots\xi^{i_m}\\
	&+y_i^2\xi^j\xi^k\,\frac{\partial^2 f_{i_1\dots i_m}}{\partial y^j\partial y^k}(y)\xi^{i_1}\dots\xi^{i_m}
	-2(y_i\delta^j_i- y_i\xi_i\xi^j)\frac{\partial f_{i_1\dots i_m}}{\partial y^j}(y)\xi^{i_1}\dots\xi^{i_m}.
\end{aligned}
$$
Performing the summation over $i$, we obtain
\begin{equation}
\begin{aligned}
		\sum\limits_{i=1}^n\frac{\partial^2\psi}{\partial\xi_i^2}\Big|_{T{\S}^{n-1}}&=m(m-1)\delta^{pq}f_{pqi_1\dots i_{m-2}}(y)\xi^{i_1}\dots\xi^{i_{m-2}}
		-2m\,y^p\,\frac{\partial f_{pi_2\dots i_m}}{\partial y^j}(y)\xi^j\xi^{i_2}\dots\xi^{i_m}\\
		&+|y|^2\,\frac{\partial^2 f_{i_1\dots i_m}}{\partial y^j\partial y^k}(y)\xi^j\xi^k\xi^{i_1}\dots\xi^{i_m}
		-2y^p\,\frac{\partial f_{i_1\dots i_m}}{\partial y^p}(y)\xi^{i_1}\dots\xi^{i_m}.
\end{aligned}
	                                  \label{4.7}
\end{equation}

Differentiating the equality \eqref{4.2}, we see that
\begin{equation}
	y^p\frac{\partial f_{pi_2\dots i_m}}{\partial y^j}=-f_{ji_2\dots i_m}.
                                 	\label{4.8}
\end{equation}
Replacing the second term on the right-hand side of \eqref{4.7} with this expression, we obtain
\begin{equation}
\begin{aligned}
		\sum\limits_{i=1}^n\frac{\partial^2\psi}{\partial\xi_i^2}\Big|_{T{\S}^{n-1}}&=\bigg[m(m-1)\delta^{pq}f_{pqi_1\dots i_{m-2}}(y)\xi^{i_1}\dots\xi^{i_{m-2}}
		+2m\, f_{i_1\dots i_m}(y)\xi^{i_1}\dots\xi^{i_m}\\
		&+|y|^2\,\frac{\partial^2 f_{i_1\dots i_m}}{\partial y^{i_{m+1}}\partial y^{i_{m+2}}}(y)\xi^{i_1}\dots\xi^{i_{m+2}}
		-2\big(\l y,\partial_y\r f\big)_{i_1\dots i_m}\xi^{i_1}\dots\xi^{i_m}\bigg]_{T{\S}^{n-1}}.
\end{aligned}
	                                \label{4.9}
\end{equation}

Now, we compute the second term on the right-hand side of \eqref{3.20}. To this end we apply the operator $\l y,\partial_y\r$ to the equation \eqref{4.6}
$$
\begin{aligned}
	\l y,\partial_y\r\psi&=y^p\frac{\partial}{\partial y^p}\Big(f_{i_1\dots i_m}\big(y-\frac{\l y,\xi\r}{|\xi|^2}\xi\big)\xi^{i_1}\dots\xi^{i_m}\Big)\\
	&=y^p\frac{\partial f_{i_1\dots i_m}}{\partial y^j}\big(y-\frac{\l y,\xi\r}{|\xi|^2}\xi\big)\big(\delta^j_p-\frac{\xi_p\xi^j}{|\xi|^2}\big)\xi^{i_1}\dots\xi^{i_m}.
\end{aligned}
$$
Setting $|\xi|=1,\ \l y,\xi\r=0$, we obtain
\begin{equation}
	(\l y,\partial_y\r\psi)|_{T{\S}^{n-1}}=\big(\l y,\partial_y\r f\big)_{i_1\dots i_m}\xi^{i_1}\dots\xi^{i_m}.
	                                  \label{4.10}
\end{equation}

We substitute expressions \eqref{4.9} and \eqref{4.10} into \eqref{3.20}
\begin{equation}
\begin{aligned}
		\Delta_\xi\varphi=\bigg[&
		-|y|^2\,\frac{\partial^2 f_{i_1\dots i_m}}{\partial y^{i_{m+1}}\partial y^{i_{m+2}}}(y)\xi^{i_1}\dots\xi^{i_{m+2}}
		+\big(\l y,\partial_y\r f\big)_{i_1\dots i_m}\xi^{i_1}\dots\xi^{i_m}\\
		&+m(m+n-4)f_{i_1\dots i_m}(y)\xi^{i_1}\dots\xi^{i_m}
		-m(m-1)\delta^{pq}f_{pqi_1\dots i_{m-2}}(y)\xi^{i_1}\dots\xi^{i_{m-2}}
		\bigg]_{T{\S}^{n-1}}.
\end{aligned}
	                                  \label{4.11}
\end{equation}

We are going to rewrite \eqref{4.11} in terms of tensor notations. First of all, the operator of contraction with the Kronecker tensor is denoted by $j$, this operator is widely used in \cite{mb}. Thus,
$$
\delta^{pq}f_{pqi_1\dots i_{m-2}}(y)\xi^{i_1}\dots\xi^{i_{m-2}}=(jf)_{i_1\dots i_{m-2}}(y)\xi^{i_1}\dots\xi^{i_{m-2}}.
$$
Let us also introduce the temporary notation
\begin{equation}
	h_{i_1\dots i_{m+2}}(y)=\sigma(i_1\dots i_{m+2})\frac{\partial^2 f_{i_1\dots i_m}}{\partial y^{i_{m+1}}\partial y^{i_{m+2}}}(y),
                                 	\label{4.12}
\end{equation}
where $\sigma(i_1\dots i_{m+2})$ is the symmetrization. The formula \eqref{4.11} takes the form
\begin{equation}
\begin{aligned}
		\Delta_\xi\varphi=\bigg[&
		-|y|^2\,h_{i_1\dots i_{m+2}}(y)\xi^{i_1}\dots\xi^{i_{m+2}}
		+\big(\l y,\partial_y\r f\big)_{i_1\dots i_m}\xi^{i_1}\dots\xi^{i_m}\\
		&+m(m+n-4)f_{i_1\dots i_m}(y)\xi^{i_1}\dots\xi^{i_m}
		-m(m-1)(jf)_{i_1\dots i_{m-2}}(y)\xi^{i_1}\dots\xi^{i_{m-2}}
		\bigg]_{T{\S}^{n-1}}.
\end{aligned}
								\label{4.13}
\end{equation}
There is no problem with two last terms on the right-hand side of \eqref{4.13}. But first two terms are problematic.
Indeed, the radial derivative $\l y,\partial_y\r f$ should not participate in the final formula for $\Delta_\xi\varphi$.
The tensor field $h$, defined by \eqref{4.12}, is not tangential.
Our conjecture is that the sum of first two terms on the right-hand side of \eqref{4.13} can be expressed in terms of some operators sending tangential tensor fields again to tangential  fields. To realize this idea, we have to consider tangential tensor fields that do not need to be symmetric.

Let $f=(f_{i_1\dots i_m})$ be a smooth tensor field on ${\R}^n\setminus\{0\}$ which is not assumed to be symmetric. Instead of \eqref{4.2}, we assume now that $f(y)$ is orthogonal to $y$ with respect to any index, i.e.,
\begin{equation}
	\quad y^pf_{i_1\dots i_{k-1}pi_{k+1}\dots i_m}(y)=0\quad\mbox{for}\quad 1\le k\le m.
                              	\label{4.14}
\end{equation}
Such $f$ are again called tangential tensor fields.
A tangential tensor field  $f$ can be restricted to the sphere ${\S}^{n-1}_\rho=\{y\in{\R}^n\mid |y|=\rho\}$ for every $\rho>0$. For the restriction
$f|_{{\S}^{n-1}_\rho}$, we can consider the covariant derivative $\nabla(f|_{{\S}^{n-1}_\rho})$
with respect to the Levi-Civita connection on the sphere ${\S}^{n-1}_\rho$ considered as a Riemannian manifold with the metric induced by the Euclidean metric of ${\R}^n$. Then the tensor field $\nabla(f|_{{\S}^{n-1}_\rho})$, being defined on ${\S}^{n-1}_\rho$ for every $\rho>0$ and smoothly depending on $\rho$, can be again considered as a rank $m+1$ tangential tensor field on ${\R}^n\setminus\{0\}$. The latter tensor field will be denoted by $\nabla f$. Let us compute coordinates of $\nabla f$. We use Cartesian coordinates on ${\R}^n$ but do not use any coordinates on spheres ${\S}^{n-1}_\rho$, this is the main idea of the current section.

For simplicity, we will do calculations in the case of $m=2$ and then will present an obvious generalization of resulting formulas for a general $m$. A second rank tensor field $f=(f_{ij})$ can be considered as the bilinear form $f(Y,Z)=f_{ij}Y^iZ^j$ on the space of vector fields. The main relation between inner geometry of a submanifold and geometry of an ambient manifold \cite{KN} is expressed in our situation as follows. Let $\nabla'$ be the Levi-Civita connection of the standard Euclidean metric on ${\R}^n$ and $\nabla$ be the Levi-Civita connection of the Riemannian metric on ${\S}^{n-1}_\rho$ induced by the Euclidean metric of ${\R}^n$. Given three smooth vector fields $X,Y,Z$ on ${\S}^{n-1}_\rho$, extend them to smooth vector fields on a neighborhood of ${\S}^{n-1}_\rho$ in ${\R}^n$ and denote the extensions by $X,Y,Z$ again. Then for a point $y\in{\S}^{n-1}_\rho$,
\begin{equation}
\begin{aligned}
\big((\nabla_{\!X}f)(Y,Z)\big)(y)&=\big((\nabla'_{\!PX}f)(PY,PZ)\big)(y)\\
&=\Big(PX\big(f(PY,PZ)\big)-f(\nabla'_{\!PX}PY,PZ)-f(PY,\nabla'_{\!PX}PZ)\Big)(y),
\end{aligned}
                                	\label{4.15}
\end{equation}
where $P:{\R}^n\rightarrow y^\bot$ is the orthogonal projection.

Choose Cartesian coordinates $(y^1,\dots, y^n)$ in ${\R}^n$ and let $\partial_i=\frac{\partial}{\partial y_i}$ be the coordinate basis. For vector fields $X=X^i\partial_i,Y=Y^i\partial_i,Z=Z^i\partial_i$,
\begin{equation}
(PX)^i=X^i-\frac{1}{|y|^2}y_pX^p\,y^i,\quad	(\nabla'_{\!X}Y)^i=X^j\frac{\partial Y^i}{\partial y^j}.
	                               \label{4.16}
\end{equation}
By \eqref{4.14}, $f(PY,PZ)=f(Y,Z)$. Therefore formula \eqref{4.15} gives
$$
\begin{aligned}
(\nabla_{\!X}f)(Y,Z)&=(PX)^k\frac{\partial}{\partial y^k}\Big(f(Y,Z)\Big)\\
&-f\big((\nabla'_{\!PX}PY)^i\partial_i,(PZ)^j\partial_j\big)-f\big((PY)^i\partial_i,(\nabla'_{\!PX}PZ)^j\partial_j\big)\\
&=(PX)^k\frac{\partial}{\partial y^k}\Big(f_{ij}(PY)^i(PZ)^j\Big)\\
&-(\nabla'_{\!PX}PY)^i(PZ)^jf_{ij}-(PY)^i(\nabla'_{\!PX}PZ)^jf_{ij}.
\end{aligned}
$$
On using \eqref{4.16}, we obtain
$$
\begin{aligned}
(\nabla_{\!X}f)(Y,Z)&
=\big(X^k-\frac{1}{|y|^2}y_pX^p\,y^k\big)\frac{\partial}{\partial y^k}\big(Y^iZ^jf_{ij}\big)\\
&-\big(X^k-\frac{1}{|y|^2}y_pX^p\,y^k\big)\Big[\frac{\partial}{\partial y^k}\big(Y^i-\frac{1}{|y|^2}y_qY^q\,y^i\big)\Big]
\big(Z^j-\frac{1}{|y|^2}y_rZ^r\,y^j\big)f_{ij}\\
&-\big(Y^i-\frac{1}{|y|^2}y_pY^p\,y^i\big)\big(X^k-\frac{1}{|y|^2}y_qX^q\,y^k\big)
\Big[\frac{\partial}{\partial y^k}\big(Z^j-\frac{1}{|y|^2}y_rZ^r\,y^j\big)\Big]f_{ij}.
\end{aligned}
$$
Using \eqref{4.14} again, we simplify this formula a little bit:
$$
\begin{aligned}
(\nabla_{\!X}f)(Y,Z)&
=\big(X^k-\frac{1}{|y|^2}y_pX^p\,y^k\big)\frac{\partial}{\partial y^k}\big(Y^iZ^jf_{ij}\big)\\
&-\big(X^k-\frac{1}{|y|^2}y_pX^p\,y^k\big)\Big[\frac{\partial}{\partial y^k}\big(Y^i-\frac{1}{|y|^2}y_qY^q\,y^i\big)\Big]
Z^jf_{ij}\\
&-Y^i\big(X^k-\frac{1}{|y|^2}y_qX^q\,y^k\big)
\Big[\frac{\partial}{\partial y^k}\big(Z^j-\frac{1}{|y|^2}y_rZ^r\,y^j\big)\Big]f_{ij}.
\end{aligned}
$$
After implementing differentiations, some terms cancel each other, and we obtain
$$
\begin{aligned}
(\nabla_{\!X}f)(Y,Z)&
=\big(X^k-\frac{1}{|y|^2}y_pX^p\,y^k\big)Y^iZ^j\frac{\partial f_{ij}}{\partial y^k}\\
&+\big(X^k-\frac{1}{|y|^2}y_pX^p\,y^k\big)\Big[\frac{\partial}{\partial y^k}\Big(\frac{1}{|y|^2}y_qY^q\,y^i\Big)\Big]
Z^jf_{ij}\\
&+Y^i\big(X^k-\frac{1}{|y|^2}y_qX^q\,y^k\big)
\Big[\frac{\partial}{\partial y^k}\Big(\frac{1}{|y|^2}y_rZ^r\,y^j\Big)\Big]f_{ij}.
\end{aligned}
$$
We write this in the form
$$
\begin{aligned}
(\nabla_{\!X}f)(Y,Z)&
=\big(X^k-\frac{1}{|y|^2}y_pX^p\,y^k\big)Y^iZ^j\frac{\partial f_{ij}}{\partial y^k}\\
&+\big(X^k-\frac{1}{|y|^2}y_pX^p\,y^k\big)\Big[\frac{1}{|y|^2}y_qY^q\delta^i_k+\frac{\partial}{\partial y^k}\Big(\frac{1}{|y|^2}y_qY^q\Big)y^i\Big]
Z^jf_{ij}\\
&+Y^i\big(X^k-\frac{1}{|y|^2}y_qX^q\,y^k\big)
\Big[\frac{1}{|y|^2}y_rZ^r\delta^j_k+\frac{\partial}{\partial y^k}\Big(\frac{1}{|y|^2}y_rZ^r\Big)y^j\Big]f_{ij}.
\end{aligned}
$$
Again using \eqref{4.14}, this is simplified to the following formula
$$
\begin{aligned}
(\nabla_{\!X}f)(Y,Z)&
=\big(X^k-\frac{1}{|y|^2}y_pX^p\,y^k\big)Y^iZ^j\frac{\partial f_{ij}}{\partial y^k}\\
&+\big(X^k-\frac{1}{|y|^2}y_pX^p\,y^k\big)\frac{1}{|y|^2}y_qY^q\delta^i_kZ^jf_{ij}\\
&+Y^i\big(X^k-\frac{1}{|y|^2}y_qX^q\,y^k\big)
\frac{1}{|y|^2}y_rZ^r\delta^j_kf_{ij}.
\end{aligned}
$$
We perform the contractions with the Kronecker tensor in two last terms
$$
\begin{aligned}
(\nabla_{\!X}f)(Y,Z)&
=\big(X^k-\frac{1}{|y|^2}y_pX^p\,y^k\big)Y^iZ^j\frac{\partial f_{ij}}{\partial y^k}\\
&+\big(X^k-\frac{1}{|y|^2}y_pX^p\,y^k\big)\frac{1}{|y|^2}y_qY^qZ^jf_{kj}\\
&+Y^i\big(X^k-\frac{1}{|y|^2}y_qX^q\,y^k\big)
\frac{1}{|y|^2}y_rZ^rf_{ik}
\end{aligned}
$$
and again simplify with the help of \eqref{4.14}
$$
(\nabla_{\!X}f)(Y,Z)=\big(X^k-\frac{1}{|y|^2}y_pX^p\,y^k\big)Y^iZ^j\frac{\partial f_{ij}}{\partial y^k}
+\frac{1}{|y|^2}y_qX^kY^qZ^jf_{kj}+\frac{1}{|y|^2}y_rX^kY^iZ^rf_{ik}.
$$
This can be written as
$$
\begin{aligned}
(\nabla_{\!X}f)(Y,Z)&=X^kY^iZ^j\frac{\partial f_{ij}}{\partial y^k}-\frac{1}{|y|^2}y_pX^pY^iZ^j\l y,\partial_y\r f_{ij}\\
&+\frac{1}{|y|^2}y_qX^kY^qZ^jf_{kj}+\frac{1}{|y|^2}y_rX^kY^iZ^rf_{ik}.
\end{aligned}
$$
Changing notations of summation indices, we write this in the form
$$
(\nabla_{\!X}f)(Y,Z)=X^kY^iZ^j\Big(\frac{\partial f_{ij}}{\partial y^k}
+\frac{y_i}{|y|^2}f_{kj}+\frac{y_j}{|y|^2}f_{ik}
-\frac{y_k}{|y|^2}\l y,\partial_y\r f_{ij}\Big).
$$
This means that
$$
\nabla_{\!k}f_{ij}=\frac{\partial f_{ij}}{\partial y^k}
+\frac{y_i}{|y|^2}f_{kj}+\frac{y_j}{|y|^2}f_{ik}
-\frac{y_k}{|y|^2}\l y,\partial_y\r f_{ij}.
$$
This formula has the obvious generalization to tensor fields of arbitrary rank
\begin{equation}
\nabla_{\!k}f_{i_1\dots i_m}=\frac{\partial f_{i_1\dots i_m}}{\partial y^k}
+\sum\limits_{a=1}^m\frac{y_{i_a}}{|y|^2}f_{i_1\dots i_{a-1}ki_{a+1}\dots i_m}
-\frac{y_k}{|y|^2}\l y,\partial_y\r f_{i_1\dots i_m}.
	                               \label{4.17}
\end{equation}
The proof is actually the same.

We emphasize that formula \eqref{4.17} is proved for a tangential tensor field $f=(f_{i_1\dots i_m})$ which is not assumed to be symmetric. As is seen from \eqref{4.17}, $\nabla f$ is again a tangential tensor field. The latter fact was mentioned above.

Next, we are going to derive a similar formula for second order covariant derivatives of a tangential tensor field. This is actually an iteration of formula \eqref{4.17}. We again consider the case of $m=2$. Let $f=(f_{ij})$ be a tangential tensor field which is not assumed to be symmetric.
Applying formula \eqref{4.17} to the third rank tangential tensor field $\nabla f$, we have
$$
\begin{aligned}
	\n_{\!l}\n_{\!k} f_{ij}
	&=\frac{\PD }{\PD y^{l}}\n_{\!k} f_{ij}+\frac{y_{k}}{|y|^2}  \n_{\!l} f_{ij}
	+\frac{y_{i}}{|y|^2}  \n_{\!k} f_{lj}+\frac{y_{j}}{|y|^2}  \n_{\!k} f_{il}
	-\frac{y_{l}}{|y|^2} \langle y,\PD_y\rangle \n_{\!k} f_{ij}.
\end{aligned}
$$
Substitute the expression \eqref{4.17} for the first term on the right-hand side
$$
\begin{aligned}
	\n_{\!l}\n_{\!k} f_{ij}
	&=\frac{\PD }{\PD y^{l}}\Big(\frac{\partial f_{ij}}{\partial y^k}
+\frac{y_i}{|y|^2}f_{kj}+\frac{y_j}{|y|^2}f_{ik}
-\frac{y_k}{|y|^2}\l y,\partial_y\r f_{ij}\Big)\\
&+\frac{y_{k}}{|y|^2}  \n_{\!l} f_{ij}
	+\frac{y_{i}}{|y|^2}  \n_{\!k} f_{lj}+\frac{y_{j}}{|y|^2}  \n_{\!k} f_{il}
	-\frac{y_{l}}{|y|^2} \langle y,\PD_y\rangle \n_{\!k} f_{ij}.
\end{aligned}
$$
Implementing the differentiation, we obtain the final formula
$$
\begin{aligned}
	\n_{\!l}\n_{\!k} f_{ij}
	&=\frac{\partial^2f_{ij}}{\partial y^k\partial y^l}
+\frac{y_i}{|y|^2}\frac{\partial f_{kj}}{\partial y^l}
+\frac{y_j}{|y|^2}\frac{\partial f_{ik}}{\partial y^l}
-\frac{y_k}{|y|^2}\l y,\partial_y\r \frac{\partial f_{ij}}{\partial y^l}\\
&+\frac{1}{|y|^2}\delta_{il}f_{kj}
+\frac{1}{|y|^2}\delta_{jl}f_{ik}
-\frac{2y_iy_l}{|y|^4}f_{kj}
-\frac{2y_jy_l}{|y|^4}f_{ik}\\
&-\frac{1}{|y|^2}\delta_{kl}\l y,\partial_y\r f_{ij}
-\frac{y_k}{|y|^2}\frac{\partial f_{ij}}{\partial y^l}
+\frac{2y_ky_l}{|y|^4}\l y,\partial_y\r f_{ij}\\
&+\frac{y_{k}}{|y|^2}  \n_{\!l} f_{ij}
	+\frac{y_{i}}{|y|^2}  \n_{\!k} f_{lj}+\frac{y_{j}}{|y|^2}  \n_{\!k} f_{il}
	-\frac{y_{l}}{|y|^2} \langle y,\PD_y\rangle \n_{\!k} f_{ij}.
\end{aligned}
$$
This formula has the obvious generalization to tangential tensor fields of arbitrary rank
\begin{equation}
\begin{aligned}
	\n_{\!k_1}\!\n_{\!k_2} f_{i_1\dots i_m}
	&=\frac{\PD^2 f_{i_1\dots i_m}}{\PD y^{k_1}\PD y^{k_2}}
	+\frac{1}{|y|^2}\sum\limits_{a=1}^m y_{i_a} \frac{\PD f_{i_1\dots i_{a-1}k_2i_{a+1}\dots i_m}}{\PD y^{k_1}}
	-\frac{1}{|y|^2} y_{k_2}\langle y,\PD_y\rangle \frac{\PD f_{i_1\dots i_m}}{\PD y^{k_1}}\\
	&+\frac{1}{|y|^2}\sum\limits_{a=1}^m \delta_{i_a k_1} f_{i_1\dots i_{a-1}k_2i_{a+1}\dots i_m}
	-\frac{2}{|y|^4}y_{k_1}\sum\limits_{a=1}^m y_{i_a} f_{i_1\dots i_{a-1}k_2i_{a+1}\dots i_m}\\
	&+\frac{2}{|y|^4} y_{k_2}y_{k_2}\langle y,\PD_y\rangle f_{i_1\dots i_m}
	-\frac{1}{|y|^2} \delta_{k_1k_2}\langle y,\PD_y\rangle f_{i_1\dots i_m}
	-\frac{1}{|y|^2} y_{k_2} \frac{\PD f_{i_1\dots i_m}}{\PD y^{k_1}}\\
	&+\frac{1}{|y|^2} y_{k_2} \n_{\!k_1} f_{i_1\dots i_m}
	+\frac{1}{|y|^2}\sum\limits_{a=1}^m y_{i_a} \n_{\!k_2} f_{i_1\dots i_{a-1}k_1i_{a+1}\dots i_m}
	-\frac{1}{|y|^2} y_{k_1}\langle y,\PD_y\rangle \n_{\!k_2} f_{i_1\dots i_m}.
\end{aligned}
	                               \label{4.18}
\end{equation}

Let us rewrite \eqref{4.18} in the form
$$
	\n_{\!k_1}\!\n_{\!k_2} f_{i_1\dots i_m}
	=\frac{\PD^2 f_{i_1\dots i_m}}{\PD y^{k_1}\PD y^{k_2}}
	+\frac{1}{|y|^2}\sum\limits_{a=1}^m \delta_{i_a k_1} f_{i_1\dots i_{a-1}k_2i_{a+1}\dots i_m}
		-\frac{1}{|y|^2} \delta_{k_1k_2}\langle y,\PD_y\rangle f_{i_1\dots i_m} +\dots,
$$
where dots stand fore some sum of terms containing at least one factor from the list $y_{i_1},\dots, y_{i_{m+2}}$.
We express second order partial derivatives from this
\begin{equation}
\begin{aligned}
	\frac{\PD^2 f_{i_1\dots i_m}}{\PD y^{i_{m+1}}\PD y^{i_{m+2}}}
	&=\n_{\!i_{m+1}}\!\n_{\!i_{m+2}} f_{i_1\dots i_m}
	-\frac{1}{|y|^2}\sum\limits_{a=1}^m \delta_{i_a i_{m+1}} f_{i_1\dots i_{a-1}i_{m+2}i_{a+1}\dots i_m}\\
	&+\frac{1}{|y|^2} \delta_{i_{m+1}i_{m+2}}\langle y,\PD_y\rangle f_{i_1\dots i_m}+\dots.
\end{aligned}
	                               \label{4.19}
\end{equation}

Recall that the Schwartz space ${\mathcal S}_\top({\R}^n;S^m{R}^n)$ of symmetric rank $m$ tangential tensor fields was introduced after formula \eqref{4.2}. Along with the latter space we will use the space $C^\infty_\top({\R}^n\setminus\{0\};S^m{R}^n)$ consisting of smooth symmetric rank $m$ tensor fields on ${\R}^n\setminus\{0\}$ satisfying \eqref{4.2}. The domain ${\R}^n\setminus\{0\}$ is foliated into spheres centered at the origin
$$
{\R}^n\setminus\{0\}=\bigcup\limits_{\rho>0}{\S}^{n-1}_\rho
$$
and the covariant derivative in \eqref{4.19} is understood in the sense of Riemannian geometry of the spheres, as is explained after formula \eqref{4.14}. The first order differential operator
\begin{equation}
d:C^\infty_\top({\R}^n\setminus\{0\};S^m{R}^n)\rightarrow C^\infty_\top({\R}^n\setminus\{0\};S^{m+1}{R}^n)
	                               \label{4.20}
\end{equation}
defined by
\begin{equation}
(df)_{i_1\dots i_{m+1}}=\sigma(i_1\dots i_{m+1})(\nabla_{\!i_1}f_{i_2\dots i_{m+1}})
	                               \label{4.21}
\end{equation}
is called the {\it inner derivative}. Actually this operator is defined on any Riemannian manifold and is widely used in integral geometry of tensor fields \cite{mb}. But in this paper the operator is always understood in the sense of spheres ${\S}^{n-1}_\rho$.

Now, we substitute the expression \eqref{4.19} into the formula \eqref{4.12}
$$
\begin{aligned}
	h_{i_1\dots i_{m+2}}(y)&=(d^2 f)_{i_1\dots i_{m+2}}(y)
	+\frac{1}{|y|^2}\sigma(i_1\dots i_{m+2})\big(\delta_{i_1i_2}(\langle y,\PD_y\rangle f)_{i_3\dots i_{m+2}}(y)\big)\\
	&-\frac{m}{|y|^2}\sigma(i_1\dots i_{m+2})\big(\delta_{i_1i_2} f_{i_3\dots i_{m+2}}(y)\big)+\dots.
\end{aligned}
$$
Insert this expression into \eqref{4.13}. The terms denoted by dots disappear since $\l y,\xi\r=0$ on $T{\S}^{n-1}$ and we obtain
$$
\begin{aligned}
	\Delta_\xi\varphi=\bigg[&
	-|y|^2\,(d^2 f)_{i_1\dots i_{m+2}}(y)\xi^{i_1}\dots\xi^{i_{m+2}}
	-\delta_{i_1i_2}(\langle y,\PD_y\rangle f)_{i_3\dots i_{m+2}}(y)\xi^{i_1}\dots\xi^{i_{m+2}}\\
	&+m\delta_{i_1i_2} f_{i_3\dots i_{m+2}}(y)\xi^{i_1}\dots\xi^{i_{m+2}}
	+\big(\l y,\partial_y\r f\big)_{i_1\dots i_m}\xi^{i_1}\dots\xi^{i_m}\\
	&+m(m+n-4)f_{i_1\dots i_m}(y)\xi^{i_1}\dots\xi^{i_m}
	-m(m-1)(jf)_{i_1\dots i_{m-2}}(y)\xi^{i_1}\dots\xi^{i_{m-2}}
	\bigg]_{T{\S}^{n-1}}.
\end{aligned}
$$
After obvious simplifications, this becomes
\begin{equation}
	\begin{aligned}
		\Delta_\xi\varphi=\bigg[&
		-|y|^2\,(d^2 f)_{i_1\dots i_{m+2}}(y)\xi^{i_1}\dots\xi^{i_{m+2}}
		+m(m+n-3)f_{i_1\dots i_m}(y)\xi^{i_1}\dots\xi^{i_m}\\
		&-m(m-1)(jf)_{i_1\dots i_{m-2}}(y)\xi^{i_1}\dots\xi^{i_{m-2}}
		\bigg]_{T{\S}^{n-1}}.
	\end{aligned}
	                           \label{4.22}
\end{equation}
The most important feature of the formula is the absence of the radial derivative $\l y,\partial_y\r$. Recall that
$\varphi=[f_{i_1\dots i_m}(y)\xi^{i_1}\dots\xi^{i_m}]_{T{\S}^{n-1}}$.

Next we consider higher powers of the operator $\Delta_\xi$.

\begin{proposition} \label{P4.1}
Let us consider $d^2,j$ and $|y|^2$ as variables of degrees $2$, $-2$ and $0$ respectively. Assume that $d^2$ and $j$ do not commute while $|y|^2$ commutes with $d^2$ and $j$. Given integers $r\ge0$ and $m\ge0$, there exist homogeneous polynomials $P^{(r,k)}(|y|^2d^2,j)\ (-r\le k\le r)$ of degree $2k$ with integer coefficients such that the equality
\begin{equation}
		\Delta^r_\xi\big[f_{i_1\dots i_{m}}(y)\xi^{i_1}\dots\xi^{i_{m}}\big]_{T{\S}^{n-1}}=
		\sum\limits_{k=-r}^{r}\big[\big(P^{(r,k)}(|y|^2d^2,j)f\big)_{i_1\dots i_{m+2k}}(y)\xi^{i_1}\dots\xi^{i_{m+2k}}\big]_{T{\S}^{n-1}}
		                                             \label{4.23}
\end{equation}
	holds for any tensor field $f\in C_\top^\infty({\R}^n\setminus\{0\};S^m{\R}^n)$. The polynomials
	$P^{(r,k)}(|y|^2d^2,j)\ (-r\le k\le r)$ are defined by the recurrent relations
\begin{equation}
		P^{(0,0)}(|y|^2d^2,j)=1
		                                             \label{4.24}
\end{equation}
	and
\begin{equation}
\begin{aligned}
			P^{(r+1,k)}(|y|^2d^2,j)
			=&-|y|^2d^2P^{(r,k-1)}(|y|^2d^2,j)\\
			&+(m\!+\!2k)(m\!+\!n\!+\!2k\!-\!3)P^{(r,k)}(|y|^2d^2,j)\\
			&-(m\!+\!2k\!+\!2)(m\!+\!2k\!+\!1)jP^{(r,k+1)}(|y|^2d^2,j),
\end{aligned}
		                                              \label{4.25}
	\end{equation}
	where it is assumed that $P^{(r,k)}=0$ for $|k|>r$.
\end{proposition}

\begin{proof}
We emphasize that the polynomials $P^{(r,k)}$ depend on $(m,n)$ although the dependence is not designated explicitly.

	We prove \eqref{4.23}--\eqref{4.25} by induction in $r$. For $r=0$, \eqref{4.23} holds tautologically (the left- and right-hand sides coincide). Assume \eqref{4.23} to be valid for some $r$. Apply the operator $\Delta_\xi$ to \eqref{4.23}
\begin{equation}
		\Delta^{r+1}_\xi\big[f_{i_1\dots i_{m}}(y)\xi^{i_1}\dots\xi^{i_{m}}\big]_{T{\S}^{n-1}}=
		\sum\limits_{k=-r}^{r}\Delta_\xi\big[(P^{(r,k)}f)_{i_1\dots i_{m+2k}}(y)\xi^{i_1}\dots\xi^{i_{m+2k}}\big]_{T{\S}^{n-1}}
		                                              \label{4.26}
\end{equation}
	For brevity we write $P^{(r,k)}$ instead of $P^{(r,k)}(|y|^2d^2,j)$.
	The tensor field $P^{(r,k)}f$ of rank $m+2k$ is also a symmetric tangential tensor field. By \eqref{4.22},
$$
\begin{aligned}
		\Delta_\xi\big[(&P^{(r,k)}f)_{i_1\dots i_{m+2k}}(y)\xi^{i_1}\dots\xi^{i_{m+2k}}\big]_{T{\S}^{n-1}}=\\
		=\bigg[&
		-|y|^2\,(d^2P^{(r,k)}f)_{i_1\dots i_{m+2k+2}}(y)\xi^{i_1}\dots\xi^{i_{m+2k+2}}\\
		&+(m\!+\!2k)(m\!+\!n\!+\!2k\!-\!3)(P^{(r,k)}f)_{i_1\dots i_{m+2k}}(y)\xi^{i_1}\dots\xi^{i_{m+2k}}\\
		&-(m\!+\!2k)(m\!+\!2k\!-\!1)(jP^{(r,k)}f)_{i_1\dots i_{m+2k-2}}(y)\xi^{i_1}\dots\xi^{i_{m+2k-2}}
		\bigg]_{T{\S}^{n-1}}.
\end{aligned}
$$
	Substitute this expression into \eqref{4.26}
$$
\begin{aligned}
		\Delta^{r+1}_\xi\big[&f_{i_1\dots i_{m}}(y)\xi^{i_1}\dots\xi^{i_{m}}\big]_{T{\S}^{n-1}}=\\
		=\sum\limits_{k=-r}^{r}\bigg[&
		-|y|^2\,(d^2P^{(r,k)}f)_{i_1\dots i_{m+2k+2}}(y)\xi^{i_1}\dots\xi^{i_{m+2k+2}}\\
		&+(m\!+\!2k)(m\!+\!n\!+\!2k\!-\!3)(P^{(r,k)}f)_{i_1\dots i_{m+2k}}(y)\xi^{i_1}\dots\xi^{i_{m+2k}}\\
		&-(m\!+\!2k)(m\!+\!2k\!-\!1)(jP^{(r,k)}f)_{i_1\dots i_{m+2k-2}}(y)\xi^{i_1}\dots\xi^{i_{m+2k-2}}
		\bigg]_{T{\S}^{n-1}}.
\end{aligned}
$$
	On the right-hand side, we group together polynomials of the same degree in $\xi$. The formula becomes
$$
\begin{aligned}
		\Delta^{r+1}_\xi\big[&f_{i_1\dots i_{m}}(y)\xi^{i_1}\dots\xi^{i_{m}}\big]_{T{\S}^{n-1}}=\\
		=\sum\limits_{k=-r-1}^{r+1}\bigg\{
		\Big[\Big(&-|y|^2\,d^2P^{(r,k-1)}
		+(m\!+\!2k)(m\!+\!n\!+\!2k\!-\!3)P^{(r,k)}\\
		&-(m\!+\!2k\!+\!2)(m\!+\!2k\!+\!1)jP^{(r,k+1)}\Big)f\Big]_{i_1\dots i_{m+2k}}(y)\xi^{i_1}\dots\xi^{i_{m+2k}}
		\bigg\}_{T{\S}^{n-1}}.
\end{aligned}
$$
	This finishes the induction step.
\end{proof}

\section{Higher order Reshetnyak formulas}\label{S:MR}

Recall that the unit sphere ${\S}^{n-1}$ is considered as a Riemannian manifold with the Riemannian metric induced by the Euclidean metric of ${\R}^n$.
Let $\tau'_{{\S}^{n-1}}$ be the cotangent bundle and $S^m\tau'_{{\S}^{n-1}}$ be the (complex) vector bundle of rank $m$ symmetric covariant tensors. There is a natural Hermitian dot-product in fibers, therefore $S^m\tau'_{{\S}^{n-1}}$ is a Hermitian vector bundle. The action of the orthogonal group $O(n)$ on ${\S}^{n-1}$ extends to the action on $S^m\tau'_{{\S}^{n-1}}$ by automorphisms of the Hermitian vector bundle.

The space $C^\infty(S^m\tau'_{{\S}^{n-1}})$ of smooth sections of $S^m\tau'_{{\S}^{n-1}}$ is the space of rank $m$ symmetric tensor fields on the sphere. The Hermitian dot-product of $S^m\tau'_{{\S}^{n-1}}$ defines $L^2$-product on the space $C^\infty(S^m\tau'_{{\S}^{n-1}})$, so it makes sense to speak of adjoint operators as well as of the action of the orthogonal group on $C^\infty(S^m\tau'_{{\S}^{n-1}})$.

We widely use two algebraic operators
$$
i:S^m\tau'_{{\S}^{n-1}}\rightarrow S^{m+2}\tau'_{{\S}^{n-1}},\quad j:S^{m+2}\tau'_{{\S}^{n-1}}\rightarrow S^{m}\tau'_{{\S}^{n-1}}
$$
of symmetric multiplication by the metric tensor and of contraction with the metric tensor. The operators $i$ and $j$ are adjoint to each other. We also use two first order differential operators
$$
d:C^\infty(S^m\tau'_{{\S}^{n-1}})\rightarrow C^\infty(S^{m+1}\tau'_{{\S}^{n-1}}),\quad
\delta:C^\infty(S^{m+1}\tau'_{{\S}^{n-1}})\rightarrow C^\infty(S^{m}\tau'_{{\S}^{n-1}}).
$$
The inner derivative $d$ is defined in local coordinates by \eqref{4.21} where $\nabla$ stands for the covariant derivative with respect to the Levi-Chivita connection on ${\S}^{n-1}$. The {\it divergence} $\delta$ is defined in local coordinates by
$$
(\delta f)_{i_1\dots i_m}=g^{pq}\nabla_{\!p}f_{qi_1\dots i_m},
$$
where $(g^{pq})=(g_{pq})^{-1}$ and $(g_{pq})$ is the metric tensor.
The operators $d$ and $-\delta$ are adjoint to each other. Each of $i,j,d,\delta$ is an invariant operator, i.e., commutes with the action of the orthogonal group.

Recall that the spaces $H^{(r,s)}_{t}(T{\S}^{n-1})$ were introduced by Definition \ref{D3.1}.

\begin{theorem} \label{Th5.1}
	Given integers $m\ge0,n\ge2$ and $r\ge0$, there exist self-adjoint linear differential operators
\begin{equation}
		A^{(m,r,l)}:C^\infty(S^m\tau'_{{\S}^{n-1}})\rightarrow C^\infty(S^m\tau'_{{\S}^{n-1}})\quad (0\le l\le r)
	                                   	\label{5.1}
\end{equation}
	such that the equality
\begin{equation}
		\|If\|^2_{H^{(r,s+1/2)}_{t+1/2}(T{\S}^{n-1})}
		=\sum\limits_{l=0}^r\int\limits_0^\infty\rho^{2t+2l+n-1}(1+\rho^2)^{s-t}
		\int\limits_{{\S}^{n-1}}\l A^{(m,r,l)}\wh f,\wh f\r(\rho\xi)\,d\xi d\rho
                                 		\label{5.2}
\end{equation}
	holds for any real $s, t>-n/2$ and for any solenoidal tensor field field $f\in{\mathcal S}_{sol}({\R}^n;S^m{\R}^n)$.

	The operators $A^{(m,r,l)}$ can be expressed as polynomials of (non-commuting) variables $(i,j,d,\delta)$ with real coefficients depending on $(m,n,r,l)$.
	The polynomials can be obtained by some recurrent procedure that will be presented below.
	In particular, each of $A^{(m,r,l)}$ is an invariant operator, i.e., commutes with the action of the orthogonal group.
	Every $A^{(m,r,l)}$ is a homogeneous differential operator of order $2l$; more precisely, $A^{(m,r,l)}$ can be written as a homogeneous polynomial of degree $2l$ of two (non-commuting) variables $d$ and $\delta$ with coefficients depending on $i$ and $j$ (the coefficients not always commute with each other as well as with $d$ and $\delta$).
	
	For every $\rho>0$, $\sum_{l=0}^r\rho^lA^{(m,r,l)}$
	is a positive operator. In particular, $A^{(m,r,0)}$ and $A^{(m,r,r)}$ are positive operators.
\end{theorem}

If the right-hand side of \eqref{5.2} is equal to zero for $f\in{\mathcal S}_{sol}({\R}^n;S^m{\R}^n)$, then $f=0$ since a solenoidal tensor field is uniquely determined by its ray transform \cite[Theorem 2.12.2]{mb}. Therefore Theorem \ref{Th5.1} suggests the following definition

\begin{definition} \label{D5.1}
For an integer $r\ge0$, real $s$ and $t>-n/2$, define the norm on the space ${\mathcal S}_{sol}({\R}^n;S^m{\R}^n)$
\begin{equation}
\|f\|^2_{H^{(r,s)}_{t,sol}({\R}^n;S^m{\R}^n)}
		=\sum\limits_{l=0}^r\int\limits_0^\infty\rho^{2t+2l+n-1}(1+\rho^2)^{s-t}
		\int\limits_{{\S}^{n-1}}\l A^{(m,r,l)}\wh f,\wh f\r(\rho\xi)\,d\xi d\rho
                                 		\label{5.3}
\end{equation}
and let the Hilbert space $H^{(r,s)}_{t,sol}({\R}^n;S^m{\R}^n)$ be the completion of
${\mathcal S}_{sol}({\R}^n;S^m{\R}^n)$ with respect to the norm \eqref{5.3}.
\end{definition}

The following theorem is the main result of our paper. It is actually an easy corollary of Theorem \ref{Th5.1}.

\begin{theorem} \label{Th5.2}
For an integer $r\ge0$, real $s$ and $t>-n/2$, the ray transform
$$
I:{\mathcal S}_{sol}({\R}^n;S^m{\R}^n)\rightarrow {\mathcal S}(T{\S}^{n-1})
$$
extends to the continuous linear operator
\begin{equation}
I:H^{(r,s)}_{t,sol}({\R}^n;S^m{\R}^n)\rightarrow H^{(r,s+1/2)}_{t+1/2}(T{\S}^{n-1})
                                 		\label{5.4}
\end{equation}
and the $r$th order Reshetnyak formula
\begin{equation}
\|f\|_{H^{(r,s)}_{t,sol}({\R}^n;S^m{\R}^n)}=\|If\|_{H^{(r,s+1/2)}_{t+1/2}(T{\S}^{n-1})}
                                 		\label{5.5}
\end{equation}
holds for any $f\in H^{(r,s)}_{t,sol}({\R}^n;S^m{\R}^n)$. In particular, \eqref{5.4} is an isometric embedding of one Hilbert space to another one.
\end{theorem}

Theorem \ref{Th5.2} follows from Theorem \ref{Th5.1}. Indeed, the Reshetnyak formula \eqref{5.5} holds for $f\in{\mathcal S}_{sol}({\R}^n;S^m{\R}^n)$ by
Theorem \ref{Th5.1}. This immediately implies the existence of the continuous extension \eqref{5.4} as well as the validity of \eqref{5.5} for
$f\in H^{(r,s)}_{t,sol}({\R}^n;S^m{\R}^n)$ since both spaces in \eqref{5.4} are completions of the corresponding Schwartz spaces.

One more important corollary of Theorem \ref{Th5.1} is the following

\begin{proposition} \label{P5.1}
{\rm (1)} The space $H^{(r,s)}_{t,sol}({\R}^n;S^m{\R}^n)$ is isotropic in the following sense. For any linear orthogonal transform $U$ of ${\R}^n$, the map
${\mathcal S}_{sol}({\R}^n;S^m{\R}^n)\rightarrow{\mathcal S}_{sol}({\R}^n;S^m{\R}^n),\ f\mapsto f\circ U$ extends to an isometry of $H^{(r,s)}_{t,sol}({\R}^n;S^m{\R}^n)$.

{\rm (1)} The space $H^{(r,s)}_{t,sol}({\R}^n;S^m{\R}^n)$ is homogeneous in the following sense. For any $a\in{\R}^n$, the map
${\mathcal S}_{sol}({\R}^n;S^m{\R}^n)\rightarrow{\mathcal S}_{sol}({\R}^n;S^m{\R}^n),\ f(x)\mapsto f(x+a)$ extends to an isometry of $H^{(r,s)}_{t,sol}({\R}^n;S^m{\R}^n)$.
\end{proposition}

\begin{proof}
The first statement immediately follows from the invariance of operators $A^{(m,r,l)}$ with respect to the action of the orthogonal group mentioned in Theorem \ref{Th5.1}.

The validity of the second statement is not seen from Definition \ref{D5.1}. Nevertheless, the second statement easily follows from the Reshetnyak formula \eqref{5.5}. Indeed, if $f_a=f(x+a)$ for $f\in{\mathcal S}_{sol}({\R}^n;S^m{\R}^n)$ and $a\in{\R}^n$, then $(If_a)(x,\xi)=(If)(x+a,\xi)$ and
$\|If_a\|_{H^{(r,s+1/2)}_{t+1/2}(T{\S}^{n-1})}=\|If\|_{H^{(r,s+1/2)}_{t+1/2}(T{\S}^{n-1})}$.
\end{proof}

The rest of the section is devoted to the proof of Theorem \ref{Th5.1}.

Given $f\in{\mathcal S}_{sol}({\R}^n;S^m{\R}^n)$, we will transform the norm $\|If\|^2_{H^{(r,s+1/2)}_{t+1/2}(T{\S}^{n-1})}$ in order to express it in terms independent of the ray transform. There will be several transformation steps.

Let $\varphi=If\in{\mathcal S}(T{\S}^{n-1})$. By \eqref{4.4},
\begin{equation}
	\wh\varphi(y,\xi)=(2\pi)^{1/2}{\wh f}_{i_1\dots i_m}(y)\xi^{i_1}\dots\xi^{i_m}\quad\big((y,\xi)\in T{\S}^{n-1}\big).
                                             	\label{5.6}
\end{equation}
By Definition \ref{D3.1},
\begin{equation}
\begin{aligned}
\|If\|^2_{H^{(r,s+1/2)}_{t+1/2}(T{\S}^{n-1})}&=		\|\varphi\|^2_{H^{(r,s+1/2)}_{t+1/2}(T{\S}^{n-1})}
=\big(({\mathbf 1}+\Delta_\xi)^r\varphi,\varphi\big)_{H^{s+1/2}_{t+1/2}(T{\S}^{n-1})}\\
&=\sum\limits_{q=0}^r{r\choose q}\big(\Delta_\xi^q\varphi,\varphi\big)_{H^{s+1/2}_{t+1/2}(T{\S}^{n-1})}.
\end{aligned}
                                        		\label{5.7}
\end{equation}
We have thus to compute the scalar products
$$
(\Delta^q_\xi\varphi,\varphi)_{H^{s+1/2}_{t+1/2}(T{\S}^{n-1})}\quad(0\le q\le r).
$$

By the definition \eqref{2.8},
$$
(\Delta^q_\xi\varphi,\varphi)_{H^{s+1/2}_{t+1/2}(T{\S}^{n-1})}
=\frac{\Gamma\big(\frac{n-1}{2}\big)}{4\pi^{(n+1)/2}}
\int\limits_{{\S}^{n-1}}\int\limits_{\xi^\bot}|y|^{2t+1}(1+|y|^2)^{s-t}\wh{\Delta^q_\xi\varphi}(y,\xi)\,\overline{\wh\varphi(y,\xi)}\, d y d \xi.
$$
Since $\Delta_\xi=-\sum_i\Xi_i^2$ commutes with the Fourier transform \cite[Lemma 4.4]{KMSS}, this can be written as
$$
(\Delta^q_\xi\varphi,\varphi)_{H^{s+1/2}_{t+1/2}(T{\S}^{n-1})}
=\frac{\Gamma\big(\frac{n-1}{2}\big)}{4\pi^{(n+1)/2}}
\int\limits_{{\S}^{n-1}}\int\limits_{\xi^\bot}|y|^{2t+1}(1+|y|^2)^{s-t}(\Delta^q_\xi\wh\varphi)(y,\xi)\,\overline{\wh\varphi(y,\xi)}\, d y d \xi.
$$
Substituting the value \eqref{5.6} for $\wh\varphi$, we obtain
\begin{equation}
	\begin{aligned}
		&\frac{2\pi^{(n-1)/2}}{\Gamma\big(\frac{n-1}{2}\big)}(\Delta^q_\xi\varphi,\varphi)_{H^{s+1/2}_{t+1/2}(T{\S}^{n-1})}=\\
		&=\int\limits_{{\S}^{n-1}}\int\limits_{\xi^\bot}|y|^{2t+1}(1+|y|^2)^{s-t}
		\Delta^q_\xi\big({\wh f}_{i_1\dots i_m}(y)\xi^{i_1}\dots\xi^{i_m}\big)\,\overline{{\wh f}_{j_1\dots j_m}(y)}\,\xi^{j_1}\dots\xi^{j_m}\, d y d \xi.
	\end{aligned}
                          	\label{5.8}
\end{equation}

Since $f$ is a solenoidal tensor field, $\wh f$ is a tangential tensor field.
By Proposition \ref{P4.1},
$$
\Delta^q_\xi\big({\wh f}_{i_1\dots i_{m}}(y)\xi^{i_1}\dots\xi^{i_{m}}\big)=
\sum\limits_{k=-q}^{q}(P^{(q,k)}\wh f)_{i_1\dots i_{m+2k}}(y)\xi^{i_1}\dots\xi^{i_{m+2k}}\quad\big((y,\xi)\in T{\S}^{n-1}\big).
$$
Substituting this expression into \eqref{5.8}, we obtain
$$
\begin{aligned}
	&\frac{2\pi^{(n-1)/2}}{\Gamma\big(\frac{n-1}{2}\big)}(\Delta^q_\xi\varphi,\varphi)_{H^{s+1/2}_{t+1/2}(T{\S}^{n-1})}=\\
	&=\sum\limits_{k=-q}^{q}\int\limits_{{\S}^{n-1}}\int\limits_{\xi^\bot}|y|^{2t+1}(1+|y|^2)^{s-t}
	(P^{(q,k)}\wh f)_{i_1\dots i_{m+2k}}(y)\overline{{\wh f}_{i_{m+2k+1}\dots i_{2m+2k}}(y)}\,\xi^{i_1}\dots\xi^{i_{2m+2k}}\, d y d \xi.
\end{aligned}
$$
Changing the order of integrations with the help of \cite[Lemma 2.15.3]{mb}, we obtain
\begin{equation}
	\begin{aligned}
		&\frac{2\pi^{(n-1)/2}}{\Gamma\big(\frac{n-1}{2}\big)}(\Delta^q_\xi\varphi,\varphi)_{H^{s+1/2}_{t+1/2}(T{\S}^{n-1})}=\\
		&=\sum\limits_{k=-q}^{q}\int\limits_{{\R}^n}|y|^{2t}(1\!+\!|y|^2)^{s-t}
		(P^{(q,k)}\wh f)_{i_1\dots i_{m+2k}}(y)\overline{{\wh f}_{i_{m+2k+1}\dots i_{2m+2k}}(y)}
		\!\!\!\int\limits_{{\S}^{n-1}\cap y^\bot}\!\!\!\!\!\!\xi^{i_1}\dots\xi^{i_{2m+2k}}\, d^{n-2}\xi\,dy.
	\end{aligned}
                                    	\label{5.9}
\end{equation}
By \cite[Lemma 2.15.4]{mb},
\begin{equation}
	\int\limits_{{\S}^{n-1}\cap y^\bot}\!\!\!\!\!\!\xi^{i_1}\dots\xi^{i_{2m+2k}}\, d^{n-2}\xi
	=\frac{2\Gamma(m+k+1/2)\pi^{(n-2)/2}}{\Gamma\big(m+k+\frac{n-1}{2}\big)}\,(\varepsilon^{m+k})^{i_1\dots i_{2m+2k}}(y),
                                       	\label{5.10}
\end{equation}
where
\begin{equation}
	(\varepsilon^{m+k})^{i_1\dots i_{2m+2k}}(y)
	=\sigma(i_1\dots i_{2m+2k})\Big(\delta^{i_1i_2}-\frac{y^{i_1}y^{i_2}}{|y|^2}\Big)\dots
	\Big(\delta^{i_{2m+2k-1}i_{2m+2k}}-\frac{y^{i_{2m+2k-1}}y^{i_{2m+2k}}}{|y|^2}\Big).
                                      	\label{5.11}
\end{equation}
Substituting the expression \eqref{5.10} into \eqref{5.9}, we obtain
\begin{equation}
	\begin{aligned}
		\frac{\pi^{1/2}}{\Gamma\big(\frac{n-1}{2}\big)}(\Delta^q_\xi\varphi,\varphi)_{H^{s+1/2}_{t+1/2}(T{\S}^{n-1})}
		&=\sum\limits_{k=-q}^{q}\frac{\Gamma(m\!+\!k\!+\!1/2)}{\Gamma\big(m\!+\!k\!+\!\frac{n-1}{2}\big)}
		\int\limits_{{\R}^n}|y|^{2t}(1\!+\!|y|^2)^{s-t}\times\\
		&\times(\varepsilon^{m+k})^{i_1\dots i_{2m+2k}}(y)(P^{(q,k)}\wh f)_{i_1\dots i_{m+2k}}(y)\overline{{\wh f}_{i_{m+2k+1}\dots i_{2m+2k}}(y)}\,dy.
	\end{aligned}
                                         	\label{5.12}
\end{equation}

Observe that both tensors $\wh f(y)$ and $(P^{(l,k)}\wh f)(y)$ are orthogonal to the vector $y$ with respect to any index. Therefore we can delete the second term $y^iy^j/|y|^2$ in all factors on the right-hand side of \eqref{5.11}. In other words, the tensor field $\varepsilon^{m+k}(y)$ in the formula \eqref{5.12} can be replaced with the tensor $\delta^{m+k}$, where
\begin{equation}
	(\delta^{m+k})^{i_1\dots i_{2m+2k}}
	=\sigma(i_1\dots i_{2m+2k})\big(\delta^{i_1i_2}\dots\delta^{i_{2m+2k-1}i_{2m+2k}}\big).
                                           	\label{5.13}
\end{equation}
The formula \eqref{5.12} becomes
\begin{equation}
\begin{aligned}
		\frac{\pi^{1/2}}{\Gamma\big(\frac{n-1}{2}\big)}(\Delta^q_\xi\varphi,\varphi)_{H^{s+1/2}_{t+1/2}(T{\S}^{n-1})}
		&=\sum\limits_{k=-q}^{q}\frac{\Gamma(m\!+\!k\!+\!1/2)}{\Gamma\big(m\!+\!k\!+\!\frac{n-1}{2}\big)}
		\int\limits_{{\R}^n}|y|^{2t}(1\!+\!|y|^2)^{s-t}\times\\
		&\times(\delta^{m+k})^{i_1\dots i_{2m+2k}}(P^{(q,k)}\wh f)_{i_1\dots i_{m+2k}}(y)\overline{{\wh f}_{i_{m+2k+1}\dots i_{2m+2k}}(y)}\,dy.
\end{aligned}
                                           	\label{5.14}
\end{equation}

For $m+2k\ge 0$, we define the linear algebraic operator
\begin{equation}
	C^{(m,k)}:S^{m+2k}{\R}^n\rightarrow S^m{\R}^n
                                         	\label{5.15}
\end{equation}
by
\begin{equation}
	\l C^{(m,k)}g,h\r=(\delta^{m+k})^{i_1\dots i_{2m+2k}} g_{i_1\dots i_{m+2k}}\overline{h_{i_{m+2k+1}\dots i_{2m+2k}}}\quad
	\big(g\in S^{m+2k}{\R}^n, h\in S^{m}{\R}^n\big).
                                         	\label{5.16}
\end{equation}
Then the formula \eqref{5.14} can be written in the form
\begin{equation}
\begin{aligned}
		&\frac{\pi^{1/2}}{\Gamma\big(\frac{n-1}{2}\big)}(\Delta^q_\xi\varphi,\varphi)_{H^{s+1/2}_{t+1/2}(T{\S}^{n-1})}=\\
		&=\sum\limits_{k=-q}^{q}\frac{\Gamma(m\!+\!k\!+\!1/2)}{\Gamma\big(m\!+\!k\!+\!\frac{n-1}{2}\big)}
		\int\limits_{{\R}^n}|y|^{2t}(1\!+\!|y|^2)^{s-t}
		\l C^{(m,k)}P^{(q,k)}(|y|^2d^2,j)\wh f(y),\wh f(y)\r\,dy.
\end{aligned}
	                                     \label{5.17}
\end{equation}

Recall that $i:S^{m}{\R}^n\rightarrow S^{m+2}{\R}^n$ and $j:S^{m+2}{\R}^n\rightarrow S^{m}{\R}^n$ are the operators of symmetric multiplication by the Kronecker tensor and of contraction with the Kronecker tensor. They are defined by
$$
(if)_{i_1\dots i_{m+2}}=\sigma(i_1\dots i_{m+2})(\delta_{i_1i_2}f_{i_3\dots i_{m+2}}),\quad
(jf)_{i_1\dots i_{m}}=\delta^{pq}f_{pqi_1\dots i_{m}}.
$$

\begin{lemma} \label{L5.1}
	For integers $m\ge0$ and $k$ satisfying $m+2k\ge0$, the operator $C^{(m,k)}$ is expressed in terms of the operators $i$ and $j$ as follows:
\begin{equation}
		C^{(m,k)}= \sum\limits_{p=\mathrm{max}(0,-k)}^{[m/2]} a_{p}(m,k)\, i^p j^{p+k},
                           		\label{5.18}
\end{equation}
	where $[m/2]$ is the integer part of $m/2$ and
\begin{equation}
		a_{p}(m,k)=\frac{2^{m-2p}m!(m+k)!(m+2k)!}{(m-2p)!p!(p+k)!(2m+2k)!}.
		                           \label{5.19}
\end{equation}
\end{lemma}

\bpr
We present the proof in the case of an even $m$ only. For  odd $m$ the proof is quite similar. We write $2m$ instead of $m$ in all formulas in the proof. The proof is based on combinatorial arguments.

We rewrite \eqref{5.16} in the form
$$
\begin{aligned}
	&\l C^{(2m,k)}g,h\r=\\
	&=\Big[\sigma(i_1\dots i_{2m +2k}j_1\dots j_{2m })
	\Big(\delta^{i_1i_2}\dots\delta^{i_{2m +2k-1}i_{2m +2k}}\delta^{j_1j_2}\dots\delta^{j_{2m -1}j_{2m }}\Big)\Big]
	g_{i_1\dots i_{2m +2k}}{\bar h}_{j_1\dots j_{2m }}.
\end{aligned}
$$
After performing the symmetrization, this becomes
\begin{equation}
	\l C^{(2m ,k)}g,h\r=\frac{1}{(4m +2k)!}\sum\limits_{\pi\in\Pi_{4m +2k}}
	\delta^{\pi_1\pi_2}\dots\delta^{\pi_{4m +2k-1}\pi_{4m +2k}} g_{i_1\dots i_{2m +2k}}{\bar h}_{j_1\dots j_{2m }},
                                    	\label{5.20}
\end{equation}
where the summation is performed over all permutations $\pi=(\pi_1,\dots,\pi_{4m +2k})$ of the set $\{i_1,\dots, i_{2m +2k},j_1,\dots,j_{2m }\}$.

We write a permutation $\pi\in\Pi_{4m +2k}$ as a sequence of pairs
\begin{equation}
	\pi=\big((\pi_1,\pi_2),(\pi_3,\pi_4),\dots,(\pi_{4m +2k-1},\pi_{4m+2k})\big).
                                   	\label{5.21}
\end{equation}
Pairs can be of 3 kinds:
$$
\begin{aligned}
	\mbox{first kind}:&\quad \mbox{both elements of the pair belong to the set}\ \{j_1,\dots,j_{2m }\};\\
	\mbox{second kind}:&\quad \mbox{one element of the pair belongs to}\ \{i_1,\dots,i_{2m +2k}\}\\
	&\quad\mbox{and another element of the pair belongs to}\ \{j_1,\dots,j_{2m }\};\\
	\mbox{third kind}:&\quad \mbox{both elements of the pair belong to the set}\ \{i_1,\dots,i_{2m +2k}\}.
\end{aligned}
$$
Obviously, the number of first kind pairs in a permutation is $\leq m $. Let a permutation $\pi\in\Pi_{4m +2k}$ contain exactly $m -p$ pairs of first kind. Then $\pi$ contains also $2p$ pairs of second kind and $m +k-p$ pairs of third kind. Therefore $m +k-p\ge0$. Thus,
$$
0\le p\le\mathrm{min}(m ,m +k).
$$

We represent $\Pi_{4m +2k}$ as the disjoint union
$$
\Pi_{4m +2k}=\bigcup\limits_{p=0}^{\mathrm{min}(m ,m +k)}\Pi_{4m +2k}^p,
$$
where $\Pi_{4m +2k}^p$ consists of  permutations containing exactly $m -p$ pairs of the first kind.
The formula \eqref{5.20} is now written as
\begin{equation}
	\l C^{(2m ,k)}g,h\r=\frac{1}{(4m +2k)!}\sum\limits_{p=0}^{\mathrm{min}(m ,m +k)}\!\!\!\!\sum\limits_{\pi\in\Pi_{4m +2k}^p}
	\delta^{\pi_1\pi_2}\dots\delta^{\pi_{4m +2k-1}\pi_{4m +2k}} g_{i_1\dots i_{2m +2k}}{\bar h}_{j_1\dots j_{2m }}.
                                     	\label{5.22}
\end{equation}
All summands of the inner sum coincide. Indeed, as we have mentioned, a permutation $\pi\in\Pi_{4m +2k}^p$ contains $m -p$ first kind pairs,
$2p$ second kind pairs and $m +k-p$ third kind pairs.
Therefore
$$
\delta^{\pi_1\pi_2}\dots\delta^{\pi_{4m +2k-1}\pi_{4m +2k}} g_{i_1\dots i_{2m +2k}}{\bar h}_{j_1\dots j_{2m }}=
\l j^{m +k-p}f, j^{m -p}g\r=\l i^{m -p}j^{m +k-p}f,g\r.
$$
The last equality holds since $j^*=i$. The formula \eqref{5.22} now gives
\begin{equation}
C^{(2m ,k)}=\frac{1}{(4m +2k)!}\sum\limits_{p=0}^{\mathrm{min}(m ,m +k)}N(2m ,k;p)\, i^{m -p}j^{m +k-p},
                                     	\label{5.23}
\end{equation}
where $N(2m ,k;p)$ is the amount of elements in the set $\Pi_{4m +2k}^p$.

It remains to compute $N(2m ,k;p)$. To this end we describe the following algorithm of constructing all permutations $\pi$ of the set $\Pi_{4m +2k}^p$. We start with an {\it empty permutation}
\begin{equation}
	\pi=\big((\cdot,\cdot)_1,(\cdot,\cdot)_2,\dots,(\cdot,\cdot)_{2m +k}\big).
                                    	\label{5.24}
\end{equation}
Then we fill in all positions of the permutation in three steps.

1. We choose a subset $\{j_{\alpha_1},\dots,j_{\alpha_{2m -2p}}\}$ of the set $\{j_1,\dots,j_{2m }\}$. There are
\begin{equation}
	2m \choose 2m -2p
                                  	\label{5.25}
\end{equation}
choices. Then we order the subset to obtain
\begin{equation}
	(2m -2p)!
                                	\label{5.26}
\end{equation}
sequences of pairs $\big((j_{\alpha_1},j_{\alpha_2}),\dots,(j_{\alpha_{2m -2p-1}},j_{\alpha_{2m -2p}})\big)$.
Preserving the order of pairs as well as order of elements in each pair, we insert these pairs into the permutation \eqref{5.24}. This can be done in
\begin{equation}
	2m +k\choose m -p
                                 	\label{5.27}
\end{equation}
ways. The result of the first step is a set of {\it partially completed} permutations. Every such permutation contains $m -p$ pairs of first kind but still contains $m +k+p$ empty pairs.

2. On the second step, we insert pairs of second kind. Let $\pi$ be one of partially completed permutations obtained on the first step. Let again $\{j_{\alpha_1},\dots,j_{\alpha_{2m -2p}}\}$ be the subset of elements participating in $\pi$. Let
$\{j_{\beta_1},\dots,j_{\beta_{2p}}\}=\{j_1,\dots,j_{2m }\}\setminus\{j_{\alpha_1},\dots,j_{\alpha_{2m -2p}}\}$.
Starting with the set $\{j_{\beta_1},\dots,j_{\beta_{2p}}\}$, we create
ordered sequences of $2p$ second kind pairs. To this end we first order the set $\{j_{\beta_1},\dots,j_{\beta_{2p}}\}$; this gives
\begin{equation}
	(2p)!
                                 	\label{5.28}
\end{equation}
ordered sequences $(j_{\beta_1},\dots,j_{\beta_{2p}})$. Then, we choose a subset $\{i_{\gamma_1},\dots,i_{\gamma_{2p}}\}$ of the set  $\{i_1,\dots,i_{2m +2k}\}$; there are
\begin{equation}
	{2m +2k\choose 2p}
                                 	\label{5.29}
\end{equation}
choices. Finally, we unite each element of the sequence $(j_{\beta_1},\dots,j_{\beta_{2p}})$ with one element of the set $\{i_{\gamma_1},\dots,i_{\gamma_{2p}}\}$ into a second kind pair; this can be done in
\begin{equation}
	2^{2p}(2p)!
                                 	\label{5.30}
\end{equation}
ways. If $(j_{\beta_r},i_{\gamma_s})$ is a second kind pair, then $(i_{\gamma_s},j_{\beta_r})$ is also a second kind pair; this explains the factor $2^{2p}$ in \eqref{5.30}.

Next, preserving the order of pairs as well as order of elements in each pair, we insert created second kind pairs into the partially completed permutation $\pi$. This can be done in
\begin{equation}
	{m +k+p\choose 2p}
                                  	\label{5.31}
\end{equation}
ways since we insert $2p$ pairs to $m +k+p$ vacant positions in $\pi$. The result of the second step is a set of partially completed permutations containing $m -p$ first kind pairs and $2p$ second kind pairs. Every such permutation still contains $m +k-p$ empty pairs.

3. For every partially completed permutation $\pi$ created on the second step, we still have $2m +2k-2p$ elements of the set $\{i_1,\dots,i_{2m +2k}\}$ which do not participate in $\pi$. We just insert these elements in an arbitrary order into $m +k-p$ empty pairs of $\pi$. This can be done in
\begin{equation}
	(2m +2k-2p)!
	                                \label{5.32}
\end{equation}
ways. This finishes the algorithm.

The algorithm gives us all permutations of the set $\Pi_{4m +2k}^p$ {\it with no duplication}. Thus, the total amount $N(2m ,k;p)$ of elements of $\Pi_{4m +2k}^p$ is equal to the product of quantities \eqref{5.25}--\eqref{5.32}, i.e.,
$$
\begin{aligned}
	N(2m ,k;p)&={2m \choose 2m -2p} (2m -2p)! {2m +k\choose m -p} (2p)!{2m +2k\choose 2p}\times\\
	&\times 2^{2p}(2p)!{m +k+p\choose 2p} (2m +2k-2p)!.
\end{aligned}
$$
After obvious simplifications, this becomes
$$
N(2m ,k;p)=\frac{2^{2p}(2m )!(2m +k)!(2m +2k)!} {(2p)!(m -p)(m +k-p)!}.
$$
Substituting this value into \eqref{5.23}, we obtain
\begin{equation}
C^{(2m ,k)}=\frac{1}{(4m +2k)!}\sum\limits_{p=0}^{\mathrm{min}(m ,m +k)}
\frac{2^{2p}(2m )!(2m +k)!(2m +2k)!} {(2p)!(m -p)(m +k-p)!}\, i^{m -p}j^{m +k-p},
                                     	\label{5.33}
\end{equation}

Finally, changing the summation variable as $p=m-q$ in \eqref{5.33}, we get
$$
C^{(2m ,k)}=\frac{1}{(4m +2k)!}\sum\limits_{q=\mathrm{max(0,-k)}}^m
\frac{2^{2m-2q}(2m )!(2m +k)!(2m +2k)!} {(2m-2q)!q!(k+q)!}\, i^{m -p}j^{m +k-p}.
$$
This coincides with \eqref{5.18}--\eqref{5.19} for an even $m$.
\epr

Recall that $P^{(q,k)}(|y|^2d^2,j)$ is a homogeneous polynomial of degree $2k$ in the variables $|y|^2d^2$ and $j$ if degrees of $d^2,j,|y|^2$ are $2,-2$ and 0 respectively; the variables $d^2$ and $j$ do not commute while $|y|^2$ commutes with $d^2$ and $j$. Let us explicitly designate the dependence on $|y|^2$. To this end we represent
\begin{equation}
	P^{(q,k)}(|y|^2d^2,j)=\sum\limits_{l=0}^{|k|}|y|^{2l}P^{(q,k,l)}(d^2,j),
                                     	\label{5.34}
\end{equation}
where the polynomial $P^{(q,k,l)}(d^2,j)$ is homogeneous of degree $2l$ in $d^2$ and homogeneous of degree $2k-2l$ in $j$. Substituting the expression \eqref{5.34} into \eqref{5.17}, we obtain
$$
\begin{aligned}
	&\frac{\pi^{1/2}}{\Gamma\big(\frac{n-1}{2}\big)}(\Delta^q_\xi\varphi,\varphi)_{H^{s+1/2}_{t+1/2}(T{\S}^{n-1})}=\\
	&=\sum\limits_{k=-q}^{q}\frac{\Gamma(m\!+\!k\!+\!1/2)}{\Gamma\big(m\!+\!k\!+\!\frac{n-1}{2}\big)}\sum\limits_{l=0}^{|k|}
	\int\limits_{{\R}^n}|y|^{2(t+l)}(1\!+\!|y|^2)^{s-t}
	\l C^{(m,k)}P^{(q,k,l)}(d^2,j)\wh f(y),\wh f(y)\r\,dy.
\end{aligned}
$$
We can now change integration variables. Setting $y=\rho\xi$, we obtain
\begin{equation}
	\begin{aligned}
		&\frac{\pi^{1/2}}{\Gamma\big(\frac{n-1}{2}\big)}(\Delta^q_\xi\varphi,\varphi)_{H^{s+1/2}_{t+1/2}(T{\S}^{n-1})}=\\
		&=\sum\limits_{k=-q}^{q}\frac{\Gamma(m\!+\!k\!+\!1/2)}{\Gamma\big(m\!+\!k\!+\!\frac{n-1}{2}\big)}\sum\limits_{l=0}^{|k|}
		\int\limits_0^\infty\rho^{2(t+l)+n-1}(1\!+\!\rho^2)^{s-t}\int\limits_{{\S}^{n-1}}
		\l C^{(m,k)}P^{(q,k,l)}(d^2,j)\wh f,\wh f\r)(\rho\xi)\,d\xi d\rho.
	\end{aligned}
                                                	\label{5.35}
\end{equation}
The operators $d$ and $j$ are now understood in the sense of intrinsic geometry of the sphere ${\S}^{n-1}$ furnished by the standard Riemannian metric: $d$ is the inner derivative and $j$ is the contraction with the metric tensor.

Let us introduce the weighted $L^2$-product on the space ${\mathcal S}({\R}^n;S^m{\R}^n)$
\begin{equation}
	(f,g)_{L^{2,s}_t({\R}^n;S^m{\R}^n)}=\int\limits_{{\R}^n}|y|^{2t}(1+|y|^2)^{s-t}\l f,g\r(y)\,dy
	=\int\limits_0^\infty\rho^{2t+n-1}(1+\rho^2)^{s-t}\int\limits_{{\S}^{n-1}}\l f,g\r(\rho\xi)\,d\xi d\rho.
                                              	\label{5.36}
\end{equation}
The formula \eqref{5.35} can be written as
\begin{equation}
	\frac{\pi^{1/2}}{\Gamma\big(\frac{n-1}{2}\big)}(\Delta^q_\xi\varphi,\varphi)_{H^{s+1/2}_{t+1/2}(T{\S}^{n-1})}
	=\sum\limits_{k=-q}^{q}\frac{\Gamma(m\!+\!k\!+\!1/2)}{\Gamma\big(m\!+\!k\!+\!\frac{n-1}{2}\big)}\sum\limits_{l=0}^{|k|}
	\big(C^{(m,k)}P^{(q,k,l)}\wh f,\wh f\big)_{L^{2,s+l}_{t+l}({\R}^n;S^m{\R}^n)}.
                                              	\label{5.37}
\end{equation}

We remember that $\varphi=If$.
It makes sense to group together terms with the same value of $l$ on the right-hand side of \eqref{5.37}. We write \eqref{5.37} in the form
\begin{equation}
	(\Delta^q_\xi If,If)_{H^{s+1/2}_{t+1/2}(T{\S}^{n-1})}
	=\sum\limits_{l=0}^q\big(B^{(m,q,l)}\wh f,\wh f\big)_{L^{2,s+l}_{t+l}({\R}^n;S^m{\R}^n)},
                                          	\label{5.38}
\end{equation}
where
\begin{equation}
	B^{(m,q,l)}=\sum\limits_{-q\le k\le q,|k|\ge l}
	\frac{\Gamma\big(\frac{n-1}{2}\big)\Gamma(m\!+\!k\!+\!1/2)}{\pi^{1/2}\Gamma\big(m\!+\!k\!+\!\frac{n-1}{2}\big)}
	C^{(m,k)}P^{(q,k,l)}(d^2,j).
	                                       \label{5.39}
\end{equation}
Substituting the expression \eqref{5.39} into \eqref{5.7}, we have
\begin{equation}
	\|If\|^2_{H^{(r,s+1/2)}_{t+1/2}(T{\S}^{n-1})}
	=\sum\limits_{l=0}^r\big(\tilde A{}^{(m,r,l)}\wh f,\wh f\big)_{L^{2,s+l}_{t+l}({\R}^n;S^m{\R}^n)},
                                          	\label{5.40}
\end{equation}
where
\begin{equation}
	\tilde A{}^{(m,r,l)}=\sum\limits_{q=l}^r{r\choose q}B^{(m,q,l)}.
                                        	\label{5.41}
\end{equation}
In view of the definition \eqref{5.36}, the formula \eqref{5.39} takes the form
\begin{equation}
	\|If\|^2_{H^{(r,s+1/2)}_{t+1/2}(T{\S}^{n-1})}
	=\sum\limits_{l=0}^r
\int\limits_0^\infty\rho^{2t+2l+n-1}(1+\rho^2)^{s-t}\int\limits_{{\S}^{n-1}}
\l\tilde A{}^{(m,r,l)}\wh f,\wh f\r(\rho\xi)\,d\xi d\rho.
                                          	\label{5.42}
\end{equation}

\begin{proof}[Proof of Theorem \ref{Th5.1}]
In the general case $\tilde A{}^{(m,r,l)}$ is not a self-adjoint operator, the corresponding example will be presented in Section 6. But all
$A^{(m,r,l)}$ must be self-adjoint operators in Theorem \ref{Th5.1}. This can be achieved as follows. Applying the complex conjugation to \eqref{5.42}, we obtain
$$
	\|If\|^2_{H^{(r,s+1/2)}_{t+1/2}(T{\S}^{n-1})}
	=\sum\limits_{l=0}^r
\int\limits_0^\infty\rho^{2t+n-1}(1+\rho^2)^{s-t}\int\limits_{{\S}^{n-1}}
\l(\tilde A{}^{(m,r,l)})^*\wh f,\wh f\r(\rho\xi)\,d\xi d\rho.
$$
Taking the sum of this equality with \eqref{5.42}, we arrive to \eqref{5.2} with the self-adjoint operators
\begin{equation}
A^{(m,r,l)}=\frac{1}{2}\big(\tilde A{}^{(m,r,l)}+(\tilde A{}^{(m,r,l)})^*\big).
                                          	\label{5.43}
\end{equation}

Thus, the operators $A^{(m,r,l)}$ are defined in several steps: the recurrent relation \eqref{4.25}, formulas \eqref{5.18}--\eqref{5.19}, \eqref{5.34}, \eqref{5.39}, \eqref{5.41}, and \eqref{5.43}. These formulas constitute the algorithm for computing the operators $A^{(m,r,l)}$. We will realize the algorithm for $r=0,1,2$ in Section 7. The algorithm can be used for every $r$, but the volume of calculations grows fast with $r$.

According to the algorithm, every $A^{(m,r,l)}$ can be represented as a polynomial of (non-commuting) variables $d,i,j$. This implies that all $A^{(m,r,l)}$ are invariant operators, i.e., they commute with the action of the orthogonal group on ${\S}^{n-1}$.

As mentioned after \eqref{5.34}, $P^{(q,k,l)}(d^2,j)$ is a homogeneous polynomial of degree $2l$ in $d^2$, if the degree of $d^2$ is assumed to be equal to 2. In other words, $P^{(q,k,l)}(d^2,j)$ is a homogeneous differential operator of order $2l$. The coefficients $C^{(m,k)}$ in \eqref{5.39} are pure algebraic operators. Therefore $A^{(m,r,l)}$ is a homogeneous differential operator of order $2l$ on ${\S}^{n-1}$. The divergence $\delta$ is mentioned in Theorem \ref{5.1} since the operator $\delta$ appears in commutator formulas for $d$ and $j$, see the next section.

The right-hand side of \eqref{5.2} is positive for every tensor field $f\in{\mathcal S}_{sol}({\R}^n;S^m{\R}^n)$ which is not identically equal to zero. This statement follows from \eqref{5.2} since a solenoidal tensor field is uniquely determined by its ray transform \cite[Theorem 2.12.2]{mb}. We believe that all $A^{(m,r,l)}$ are non-negative operators. This fact will be checked for  $r=0,1,2$ and for small $m$ in Section 7, but so far we cannot prove it for general $(m,r)$.

Let us now prove that $\sum_{l=0}^r\rho^lA^{(m,r,l)}$ is a positive operator for every $\rho>0$. In particular, if $\rho$ is either very small or very big, this gives the positiveness of the operators $A^{(m,r,0)}$ and $A^{(m,r,r)}$ since $A^{(m,r,l)}$ are independent of $\rho$.

In \eqref{5.2}, $f$ is an arbitrary tensor field from the space ${\mathcal S}_{sol}({\R}^n;S^m{\R}^n)$. In terms of the Fourier transform this means that $\wh f$ is an arbitrary tensor field from ${\mathcal S}_\top({\R}^n;S^m{\R}^n)$. On using the latter fact we separate variables in \eqref{5.2}, that is, choose $\wh f$ in the form $\wh f(\rho\xi)=\alpha(\rho)g(\xi)$, where $g\in C^\infty(S^m\tau'_{{\S}^{n-1}})$ is an arbitrary tensor field on the sphere and $\alpha$ is an arbitrary function from ${\mathcal S}(\R)$. For such a choice, \eqref{5.2} becomes
$$
	\|If\|^2_{H^{(r,s+1/2)}_{t+1/2}(T{\S}^{n-1})}
	=\sum\limits_{l=0}^r\int\limits_0^\infty\rho^{2t+2l+n-1}(1+\rho^2)^{s-t}\alpha^2(\rho)
	\int\limits_{{\S}^{n-1}}\l A^{(m,r,l)}g,g\r(\xi)\,d\xi d\rho.
$$
	In particular,
\begin{equation}
		\int\limits_0^\infty\rho^{2t+n-1}(1+\rho^2)^{s-t}\alpha^2(\rho)
		\Big(\int\limits_{{\S}^{n-1}}\Big\langle\sum\limits_{l=0}^r \rho^{2l}A^{(m,r,l)}g,g\Big\rangle(\xi)\,d\xi\Big) d\rho>0
		                                           \label{5.44}
\end{equation}
	for any tensor field $g\in C^\infty(S^m\tau'_{{\S}^{n-1}})$ not identically equal to zero.
	
	Let us use the arbitrariness of the function $\alpha$ in \eqref{5.44}. For a fixed $\rho_0>0$, we can choose $\alpha\in{\mathcal S}(\R)$ supported in an arbitrary neighborhood of $\rho_0$. Therefore \eqref{5.44} implies
$$
	\int\limits_{{\S}^{n-1}}\Big\langle\sum\limits_{l=0}^r \rho_0^{2l}A^{(m,r,l)}g,g\Big\rangle(\xi)\,d\xi>0.
$$
	This proves the positiveness of $\sum_{l=0}^r\rho^lA^{(m,r,l)}$ since $\rho_0>0$ is arbitrary.
\end{proof}

We finish the section with a remark that is important for applications of Theorems \ref{Th5.1} and \ref{Th5.2}. Recall \cite[Theorem 2.6.2]{mb} that a tensor $f\in{\mathcal S}({\R}^n;S^m{\R}^n)$ is uniquely represented as the sum of solenoidal and potential parts
\begin{equation}
f={}^s\!f+dv,\quad \delta\,{}^s\!f=0.
		                                           \label{5.45}
\end{equation}
The ray transform does not see the potential part, i.e., $If=I({}^s\!f)$. The tensor field ${}^s\!f$ belongs to the space $C^\infty_{sol}({\R}^n;S^m{\R}^n)$ but not to ${\mathcal S}_{sol}({\R}^n;S^m{\R}^n)$. Indeed, in the general case ${}^s\!f(x)$ decays at infinity as $(1+|x|)^{1-n}$ but does not fast decay. Therefore, formally speaking, Theorem \ref{Th5.1} does not apply to ${}^s\!f$. Nevertheless, the situation can be easily improved. Indeed, the Fourier transform $\wh{{}^s\!f}(y)$ is smooth on ${\R}^n\setminus\{0\}$, fast decays at infinity but has a singularity at $y=0$. Fortunately, $\wh{{}^s\!f}(y)$ is bounded on the whole of ${\R}^n$, i.e., the singularity concerns positive order derivatives of $\wh{{}^s\!f}(y)$ only, see \cite[Theorem 2.6.2]{mb}.
This immediately implies that ${}^s\!f$ belongs to $H^{(r,s)}_{t,sol}({\R}^n;S^m{\R}^n)$ and Theorem \ref{Th5.2} applies to ${}^s\!f$. Moreover, Theorem \ref{Th5.1} actually applies to ${}^s\!f$ too. Indeed, by \eqref{5.6}, the function
$$
\varphi(y,\xi)=\Big[{\wh{{}^s\!f}}_{i_1\dots i_m}(y)\xi^{i_1}\dots\xi^{i_m}\Big]_{T{\S}^{n-1}}=(2\pi)^{-1/2}\wh{If}
$$
belongs to ${\mathcal S}(T{\S}^{n-1})$. Recall that our proof of \eqref{5.2} is based on the usage of this function. Thus, no singularity appears on the right-hand side of \eqref{5.2} while replacing $f$ with ${}^s\!f$ for $f\in{\mathcal S}({\R}^n;S^m{\R}^n)$.

The decomposition \eqref{5.45} is also valid for symmetric tensor fields of less regularity, see for example \cite[Theorem 3.5]{Sh3}. Theorems \ref{Th5.1} and \ref{Th5.2} with appropriate modifications apply to ${}^s\!f$ in all such cases.

\section{Reshetnyak formulas of orders $0,1,2$} \label{S:SC}

\subsection{Zeroth order Reshetnyak formula}

In the case of $r=0$, Theorem \ref{Th5.1} gives: for every real $s$ and $t>-n/2$, the equality
\begin{equation}
\|If\|^2_{H^{s+1/2}_{t+1/2}(T{\S}^{n-1})}
=\int\limits_0^\infty\rho^{2t+n-1}(1+\rho^2)^{s-t}\int\limits_{{\S}^{n-1}}\l A^{(m,0,0)}\wh f,\wh f\r(\rho\xi)\,d\xi d\rho
                                         	\label{6.1}
\end{equation}
holds for any tensor field $f\in{\mathcal S}_{sol}({\R}^n;S^m{\R}^n)$.
The operator $A^{(m,r,l)}$ is defined in Sections 4--5 by a chain of formulas and recurrent relations. Almost all these formulas are very easy in the case of $r=l=0$ and we obtain
\begin{equation}
A^{(m,0,0)}=\frac{\Gamma\big(\frac{n-1}{2}\big)}{2\pi^{(n-1)/2}}\,\sum\limits_{k=0}^{[m/2]} a_k(m,n)\, i^kj^k.
                                        	\label{6.2}
\end{equation}
Here $[m/2]$ is the integer part of $m/2$ and the coefficients are expressed by
\begin{equation}
	a_k(m,n)=\frac{2^{m+1}\pi^{(n-2)/2}(m!)^3\Gamma\big(m+\frac{1}{2}\big)}{(2m)!\Gamma\big(m+\frac{n-1}{2}\big)}\,
	\frac{1}{2^{2k}(k!)^2(m-2k)!}.
                                        	\label{6.3}
\end{equation}
This actually coincides with \cite[Theorem 4.2]{Sh3}. Nevertheless, we indicate 3 following differences between \cite[Theorem 4.2]{Sh3} and \eqref{6.2}--\eqref{6.3}.

\begin{enumerate}
	\item The factor $\frac{\Gamma\big(\frac{n-1}{2}\big)}{2\pi^{(n-1)/2}}$ is added on the right-hand side of \eqref{6.2} since the definition
$$
\|\varphi\|^2_{H^s_t(T{\S}^{n-1})}=\frac{1}{2\pi}
\int\limits_{{\S}^{n-1}}\int\limits_{\xi^\bot}|y|^{2t}(1+|y|^2)^{s-t}|\wh\varphi(y,\xi)|^2\, d y d \xi
$$
is used in \cite{Sh3} which differs by the factor $\frac{2\pi^{(n-1)/2}}{\Gamma\big(\frac{n-1}{2}\big)}$ of our definition \eqref{2.8}.
\item The factor $\pi^{(n-2)/2}$ is written on the right hand side of \eqref{6.3} instead of the factor $\pi^{(n-1)/2}$ in the formula for $a_k(m,n)$ in \cite[Theorem 4.2]{Sh3}; it is just a misprint in \cite{Sh3}, compare with \cite[formula 2.5.3]{mb}.
\item The factor $(m!)^3$ participates on the right-hand side of \eqref{6.3} although it is absent in both \cite{mb} and \cite{Sh3}; this is also a misprint (indeed, the factor $(m!)^3$ is presented in formula (4.11) of \cite{Sh3} but the factor is lost in the corresponding formula in the statement of \cite[Theorem 4.2]{Sh3}; unfortunately, the same misprint is in \cite{mb}).
\end{enumerate}

As is seen from \eqref{6.2}, $A^{(m,0,0)}$ is a positive self-adjoint operator.

\subsection{First order Reshetnyak formula}

By Theorem \ref{Th5.1}, the first order Reshetnyak formula
\begin{equation}
\begin{aligned}
		\|If\|^2_{H^{(1,s+1/2)}_{t+1/2}(T{\S}^{n-1})}&=
		\int\limits_0^\infty\rho^{2t+n+1}(1+\rho^2)^{s-t}\int\limits_{{\S}^{n-1}}\l A^{(m,1,1)}\wh f,\wh f\r(\rho\xi)\,d\xi d\rho\\
		&+\int\limits_0^\infty\rho^{2t+n-1}(1+\rho^2)^{s-t}\int\limits_{{\S}^{n-1}}\l A^{(m,1,0)}\wh f,\wh f\r(\rho\xi)\,d\xi d\rho
\end{aligned}
                            	\label{6.4}
\end{equation}
holds for every real $s$, $t>-n/2$ and for any tensor field $f\in{\mathcal S}_{sol}({\R}^n;S^m{\R}^n)$.

By Theorem \ref{Th5.1}, $A^{(m,1,1)}$ is a second order differential operator while $A^{(m,1,0)}$ is an algebraic operator. We compute these operators following the scheme presented in Sections 4--5, but in the reverse order.

First of all by \eqref{5.43},
\begin{equation}
	A^{(m,1,0)}=\frac{1}{2}\big({\tilde A}{}^{(m,1,0)}+({\tilde A}{}^{(m,1,0)})^*\big),\quad
A^{(m,1,1)}=\frac{1}{2}\big({\tilde A}{}^{(m,1,1)}+({\tilde A}{}^{(m,1,1)})^*\big)
                                    	\label{6.5}
\end{equation}

By \eqref{5.41},
\begin{equation}
	{\tilde A}{}^{(m,1,0)}=B^{(m,0,0)}+B^{(m,1,0)},\quad {\tilde A}{}^{(m,1,1)}=B^{(m,1,1)}.
                                        	\label{6.6}
\end{equation}

By \eqref{5.39},
$$
\begin{aligned}
	B^{(m,0,0)}&=\frac{\Gamma\big(\frac{n-1}{2}\big)\Gamma(m\!+\!1/2)}{\pi^{1/2}\Gamma\big(m\!+\!\frac{n-1}{2}\big)}
	C^{(m,0)}P^{(0,0,0)},\\
	B^{(m,1,0)}&=
	\frac{\Gamma\big(\frac{n-1}{2}\big)\Gamma(m\!-\!1/2)}{\pi^{1/2}\Gamma\big(m\!+\!\frac{n-3}{2}\big)}C^{(m,-1)}P^{(1,-1,0)}
	+\frac{\Gamma\big(\frac{n-1}{2}\big)\Gamma(m\!+\!1/2)}{\pi^{1/2}\Gamma\big(m\!+\!\frac{n-1}{2}\big)}C^{(m,0)}P^{(1,0,0)}\\
	&+\frac{\Gamma\big(\frac{n-1}{2}\big)\Gamma(m\!+\!3/2)}{\pi^{1/2}\Gamma\big(m\!+\!\frac{n+1}{2}\big)}C^{(m,1)}P^{(1,1,0)},\\
	B^{(m,1,1)}&=
	\frac{\Gamma\big(\frac{n-1}{2}\big)\Gamma(m\!-\!1/2)}{\pi^{1/2}\Gamma\big(m\!+\!\frac{n-3}{2}\big)}C^{(m,-1)}P^{(1,-1,1)}
	+\frac{\Gamma\big(\frac{n-1}{2}\big)\Gamma(m\!+\!3/2)}{\pi^{1/2}\Gamma\big(m\!+\!\frac{n+1}{2}\big)}C^{(m,1)}P^{(1,1,1)}.
\end{aligned}
$$
Substitute these values into \eqref{6.6}
$$
\begin{aligned}
	{\tilde A}{}^{(m,1,0)}&=\frac{\Gamma\big(\frac{n-1}{2}\big)}{\pi^{1/2}}\bigg[
	\frac{\Gamma(m\!-\!1/2)}{\Gamma\big(m\!+\!\frac{n-3}{2}\big)}C^{(m,-1)}P^{(1,-1,0)}
	+\frac{\Gamma(m\!+\!1/2)}{\Gamma\big(m\!+\!\frac{n-1}{2}\big)}C^{(m,0)}\big(P^{(1,0,0)}+{\mathbf 1}\big)\\
	&\qquad\qquad+\frac{\Gamma(m\!+\!3/2)}{\Gamma\big(m\!+\!\frac{n+1}{2}\big)}C^{(m,1)}P^{(1,1,0)}\bigg],\\
	{\tilde A}{}^{(m,1,1)}&=\frac{\Gamma\big(\frac{n-1}{2}\big)}{\pi^{1/2}}\bigg[
	\frac{\Gamma(m\!-\!1/2)}{\Gamma\big(m\!+\!\frac{n-3}{2}\big)}C^{(m,-1)}P^{(1,-1,1)}
	+\frac{\Gamma(m\!+\!3/2)}{\Gamma\big(m\!+\!\frac{n+1}{2}\big)}C^{(m,1)}P^{(1,1,1)}\bigg].
\end{aligned}
$$
We have used that $P^{(0,0,0)}$ is the identity operator ${\mathbf 1}$ as follows from Proposition \ref{P4.1} and formula \eqref{5.34}. This can be simplified a little bit:
\begin{equation}
\begin{aligned}
		{\tilde A}{}^{(m,1,0)}&=\frac{\Gamma\big(\frac{n-1}{2}\big)\Gamma(m\!+\!1/2)}{\pi^{1/2}\Gamma\big(m\!+\!\frac{n-1}{2}\big)}\bigg[
		\frac{2m\!+\!n\!-\!3}{2m\!-\!1}C^{(m,-1)}P^{(1,-1,0)}
		+C^{(m,0)}\big(P^{(1,0,0)}+{\mathbf 1}\big)\\
		&\qquad\qquad\qquad\qquad\qquad+\frac{2m\!+\!1}{2m\!+\!n\!-\!1}C^{(m,1)}P^{(1,1,0)}\bigg],\\
		{\tilde A}{}^{(m,1,1)}&=\frac{\Gamma\big(\frac{n-1}{2}\big)\Gamma(m\!+\!3/2)}{\pi^{1/2}\Gamma\big(m\!+\!\frac{n+1}{2}\big)}\bigg[
		\frac{(2m\!+\!n\!-\!1)(2m\!+\!n\!-\!3)}{(2m\!+\!1)(2m\!-\!1)}C^{(m,-1)}P^{(1,-1,1)}
		+C^{(m,1)}P^{(1,1,1)}\bigg].
\end{aligned}
                                          	\label{6.7}
\end{equation}

Using recurrent relations of Proposition \ref{P4.1}, we compute
\begin{equation}
	P^{(1,-1)}=m(m+1)j,\quad P^{(1,0)}=m(m+n-3),\quad P^{(1,1)}=-|y|^2d^2.
                                         	\label{6.8}
\end{equation}
This implies with the help of \eqref{5.34}
\begin{equation}
P^{(1,-1,0)}=m(m+1)j,\ P^{(1,-1,1)}=0,\quad P^{(1,0,0)}=m(m+n-3),\quad P^{(1,1,0)}=0,\ P^{(1,1,1)}=-d^2.
                                         	\label{6.9}
\end{equation}
Substitute these values into \eqref{6.7}
\begin{equation}
\begin{aligned}
		{\tilde A}{}^{(m,1,0)}&=\frac{\Gamma\big(\frac{n-1}{2}\big)\Gamma(m\!+\!1/2)}{\pi^{1/2}\Gamma\big(m\!+\!\frac{n-1}{2}\big)}\bigg[
		\frac{m(m\!+\!1)(2m\!+\!n\!-\!3)}{2m\!-\!1}C^{(m,-1)}j
		+\big(m(m\!+\!n\!-\!3)+1\big)C^{(m,0)}\bigg],\\
		{\tilde A}{}^{(m,1,1)}&=-\frac{\Gamma\big(\frac{n-1}{2}\big)\Gamma(m\!+\!3/2)}{\pi^{1/2}\Gamma\big(m\!+\!\frac{n+1}{2}\big)}\,C^{(m,1)}d^2.
\end{aligned}
                                           	\label{6.10}
\end{equation}

By Lemma \ref{L5.1},
\begin{equation}
	\begin{aligned}
		C^{(m,-1)}&=\sum\limits_{p=1}^{[m/2]}a_p(m,-1)i^pj^{p-1},\quad
		a_p(m,-1)=\frac{2^{m-2p}m!(m-1)!(m-2)!}{(m-2p)!p!(p-1)!(2m-2)!};\\
		C^{(m,0)}&=\sum\limits_{p=0}^{[m/2]}a_p(m,0)i^pj^{p},\quad
		a_p(m,0)=\frac{2^{m-2p}(m!)^3}{(m-2p)!(p!)^2(2m)!};\\
		C^{(m,1)}&=\sum\limits_{p=0}^{[m/2]}a_p(m,1)i^pj^{p+1},\quad
		a_p(m,1)=\frac{2^{m-2p}m!(m+1)!(m+2)!}{(m-2p)!p!(p+1)!(2m+2)!}.
	\end{aligned}
	                                          \label{6.11}
\end{equation}

Substituting values \eqref{6.11} into the first of formulas \eqref{6.10}, we have
$$
\begin{aligned}
	{\tilde A}{}^{(m,1,0)}&=\frac{\Gamma\big(\frac{n-1}{2}\big)\Gamma(m\!+\!1/2)}{\pi^{1/2}\Gamma\big(m\!+\!\frac{n-1}{2}\big)}\bigg[
	\big(m(m\!+\!n\!-\!3)+1\big)a_0(m,0)\\
	&+\sum\limits_{p=1}^{[m/2]}\Big(\frac{m(m\!+\!1)(2m\!+\!n\!-\!3)}{2m-1}a_p(m,-1)+\big(m(m\!+\!n\!-\!3)+1\big)a_p(m,0)\Big)i^pj^p\bigg].
\end{aligned}
$$
On assuming $a_0(m,-1)=0$, this can be written as
\begin{equation}
	{\tilde A}{}^{(m,1,0)}=\sum\limits_{p=0}^{[m/2]}\alpha^{(m,1,0)}_p\,i^pj^p,
	                                       \label{6.12}
\end{equation}
where
$$
\alpha^{(m,1,0)}_p=\frac{\Gamma\big(\frac{n-1}{2}\big)\Gamma(m\!+\!1/2)}{\pi^{1/2}\Gamma\big(m\!+\!\frac{n-1}{2}\big)}\Big[
\frac{m(m\!+\!1)(2m\!+\!n\!-\!3)}{2m-1}a_p(m,-1)+\big(m(m\!+\!n\!-\!3)+1\big)a_p(m,0)\Big].
$$
Substituting values \eqref{6.11} for $a_p(m,-1)$ and $a_p(m,0)$, we obtain
\begin{equation}
	\begin{aligned}
		\alpha^{(m,1,0)}_p&=\frac{\Gamma\big(\frac{n-1}{2}\big)\Gamma(m\!+\!1/2)m(m!)^2(m\!-\!2)!}
		{\pi^{1/2}\Gamma\big(m\!+\!\frac{n-1}{2}\big)(2m)!}\,\frac{2^{m-2p}}{(p!)^2(m-2p)!}\times\\
		&\times\Big[2p(m\!+\!1)(2m\!+\!n\!-\!3)+(m\!-\!1)\big(m(m\!+\!n\!-\!3)+1\big)\Big]\quad(m\ge2).
	\end{aligned}
                                           	\label{6.13}
\end{equation}
Because of the factor $(m-2)!$, this formula makes sense for $m\ge2$. We will consider the cases of $m=0$ and $m=1$ a little bit later.

Substituting the value for $C^{(m,1)}$ from \eqref{6.11} into the second of formulas \eqref{6.10}, we have
$$
{\tilde A}{}^{(m,1,1)}=-\frac{\Gamma\big(\frac{n-1}{2}\big)\Gamma(m\!+\!3/2)}{\pi^{1/2}\Gamma\big(m\!+\!\frac{n+1}{2}\big)}
\sum\limits_{p=0}^{[m/2]}a_p(m,1)\,i^pj^{p+1}d^2.
$$
This can be written as
\begin{equation}
	{\tilde A}{}^{(m,1,1)}=-\sum\limits_{p=0}^{[m/2]}\alpha^{(m,1,1)}_p\,i^pj^{p+1}d^2,
                                         	\label{6.14}
\end{equation}
where
$$
\alpha^{(m,1,1)}_p=\frac{\Gamma\big(\frac{n-1}{2}\big)\Gamma(m\!+\!3/2)}{\pi^{1/2}\Gamma\big(m\!+\!\frac{n+1}{2}\big)}\,a_p(m,1).
$$
Substituting the value \eqref{6.11} for $a_p(m,1)$, we obtain
\begin{equation}
	\alpha^{(m,1,1)}_p=\frac{\Gamma\big(\frac{n-1}{2}\big)\Gamma(m\!+\!3/2)m!(m\!+\!1)!(m\!+\!2)!}
	{\pi^{1/2}\Gamma\big(m\!+\!\frac{n+1}{2}\big)(2m+2)!}\,\frac{2^{m-2p}}{p!(p\!+\!1)!(m-2p)!}.
                                        	\label{6.15}
\end{equation}

Let us also specify the formula \eqref{6.13} for $m=0$ and $m=1$. In the case of $m=0$, the first of formulas \eqref{6.10} gives
${\tilde A}{}^{(m,1,0)}=C^{(0,0)}$. By \eqref{6.11}, $C^{(0,0)}=a_0(0,0){\mathbf 1}={\mathbf 1}$. Thus,
\begin{equation}
	{\tilde A}{}^{(m,1,0)}={\mathbf 1}.
	                                       \label{6.16}
\end{equation}
In the case of $m=1$, the first of formulas \eqref{6.10} gives
${\tilde A}{}^{(m,1,0)}=2C^{(1,-1)}j+C^{(1,0)}$. By \eqref{6.11}, $C^{(1,-1)}=0$ and $C^{(1,0)}=a_0(1,0){\mathbf 1}={\mathbf 1}$. Therefore
${\tilde A}{}^{(m,1,0)}={\mathbf 1}$. Thus, we can assume \eqref{6.12} to be valid for all $m$ with the formula \eqref{6.13} added by
\begin{equation}
		\alpha^{(0,1,0)}_0=\alpha^{(1,1,0)}_0=1.
                                           	\label{6.17}
\end{equation}

As is seen from \eqref{6.12}--\eqref{6.13}, ${\tilde A}{}^{(m,1,0)}$ is a positive self-adjoint operator.
By Theorem \ref{5.1}, $A^{(m,1,1)}=\frac{1}{2}\big({\tilde A}{}^{(m,1,1)}+({\tilde A}{}^{(m,1,1)})^*\big)$ must be a positive operator.
As far as the operator ${\tilde A}{}^{(m,1,1)}$ is concerned, its non-negativeness and self-adjointness are not obvious. Deleting factors independent of $p$ in \eqref{6.15}, we pose the following

\begin{conjecture} \label{Cj6.1}
	The second order differential operator
\begin{equation}
D^{(m)}=-\sum\limits_{p=0}^{[m/2]}\frac{2^{-2p}}{p!(p\!+\!1)!(m-2p)!}\,i^pj^{p+1}d^2:C^\infty(S^m\tau'_{{\S}^{n-1}})\rightarrow C^\infty(S^m\tau'_{{\S}^{n-1}})
                                     	\label{6.18}
\end{equation}
is a non-negative self-adjoint operator for every $m$.
\end{conjecture}

If the self-adjointness of $D^{(m)}$ was proved, then its positiveness would follow by Theorem \ref{5.1}. We will check Conjecture \ref{Cj6.1} for $m=0,1,2$. In the general case the conjecture remains unproved.

We use local coordinates on the sphere, $(g_{ij})$ is the metric tensor and $(g^{ij})=(g_{ij})^{-1}$.
In the case of $m=0$, the formula  \eqref{6.18} becomes
$D^{(0)}=-jd^2=-\Delta$,
where $\Delta=g^{ij}{\nabla}_{\!i}{\nabla}_{\!j}$ is the rough Laplacian. Conjecture \ref{Cj6.1} is true for $m=0$.

In the case of $m=1$, the formula  \eqref{6.18} becomes
$$
(D^{(1)}f)_i=-(jd^2f)_i=-\frac{1}{6}g^{jk}({\nabla}_{\!i}{\nabla}_{\!j}f_k+{\nabla}_{\!j}{\nabla}_{\!i}f_k+
{\nabla}_{\!i}{\nabla}_{\!k}f_j+{\nabla}_{\!k}{\nabla}_{\!i}f_j+
{\nabla}_{\!j}{\nabla}_{\!k}f_i+{\nabla}_{\!k}{\nabla}_{\!j}f_i).
$$
We write this as
\begin{equation}
	(D^{(1)}f)_i=-\frac{1}{3}\big((d\delta f)_i+{\nabla}_{\!p}{\nabla}_{\!i}f^p+(\Delta f)_i\big).
                                             	\label{6.19}
\end{equation}
On the other hand,
$$
(df)_{ij}=\frac{1}{2}({\nabla}_{\!i}f_j+{\nabla}_{\!j}f_i)
$$
and
$$
\begin{aligned}
	(\delta df)_i=g^{pq}{\nabla}_{\!p}(df)_{iq}&=\frac{1}{2}g^{pq}{\nabla}_{\!p}({\nabla}_{\!i}f_q+{\nabla}_{\!q}f_i)\\
	&=\frac{1}{2}g^{pq}({\nabla}_{\!p}{\nabla}_{\!i}f_q+{\nabla}_{\!p}{\nabla}_{\!q}f_i)
	=\frac{1}{2}\big({\nabla}_{\!p}{\nabla}_{\!i}f^p+(\Delta f)_i\big).
\end{aligned}
$$
From this
$$
{\nabla}_{\!p}{\nabla}_{\!i}f^p+(\Delta f)_i=2(\delta df)_i.
$$
Substituting this expression into \eqref{6.19}, we obtain
$$
D^{(1)}=-\frac{2}{3}\delta d -\frac{1}{3}d\delta.
$$	
Both $-d\delta$ and $-\delta d$ are non-negative self-adjoint operators. Thus, Conjecture \ref{Cj6.1} is true for $m=1$.

In the case of $m=2$, the formula  \eqref{6.18} becomes
\begin{equation}
D^{(2)}=-\frac{1}{2}\,jd^2-\frac{1}{8}\, ij^2d^2\quad\mbox{on}\ S^2.
                               	\label{6.20}
\end{equation}
For a symmetric tensor field $f=(f_{ij})$, we have
$$
\begin{aligned}
	(d^2f)_{ijkl}=\frac{1}{12}&\big({\nabla}_{\!i}{\nabla}_{\!j}f_{kl}+{\nabla}_{\!j}{\nabla}_{\!i}f_{kl}+
	{\nabla}_{\!i}{\nabla}_{\!k}f_{jl}+{\nabla}_{\!k}{\nabla}_{\!i}f_{jl}+
	{\nabla}_{\!i}{\nabla}_{\!l}f_{jk}+{\nabla}_{\!l}{\nabla}_{\!i}f_{jk}\\
	&+{\nabla}_{\!j}{\nabla}_{\!k}f_{il}+{\nabla}_{\!k}{\nabla}_{\!j}f_{il}+
	{\nabla}_{\!j}{\nabla}_{\!l}f_{ik}+{\nabla}_{\!l}{\nabla}_{\!j}f_{ik}+
	{\nabla}_{\!k}{\nabla}_{\!l}f_{ij}+{\nabla}_{\!l}{\nabla}_{\!k}f_{ij}\big).
\end{aligned}
$$
Contracting this equality with the metric tensor $g^{kl}$, we get
\begin{equation}
	(jd^2f)_{ij}=\frac{1}{6}{\nabla}_{\!i}{\nabla}_{\!j}f_p^p
	+\frac{1}{6}\big({\nabla}_{\!i}{\nabla}_{\!p}f_j^p+{\nabla}_{\!j}{\nabla}_{\!p}f_i^p
	+{\nabla}_{\!p}{\nabla}_{\!i}f_j^p+{\nabla}_{\!p}{\nabla}_{\!j}f_i^p+(\Delta f)_{ij}\big).
                               	\label{6.21}
\end{equation}
This can be written as
\begin{equation}
	(jd^2f)_{ij}=\frac{1}{3}(d\delta f)_{ij}
	+\frac{1}{6}\big({\nabla}_{\!p}{\nabla}_{\!i}f_j^p+{\nabla}_{\!p}{\nabla}_{\!j}f_i^p+(\Delta f)_{ij}\big)+\frac{1}{6}(d^2 j f)_{ij}.
                                    	\label{6.22}
\end{equation}
On the other hand,
$$
(df)_{ijk}=\frac{1}{3}({\nabla}_{\!i}f_{jk}+{\nabla}_{\!j}f_{ik}+{\nabla}_{\!k}f_{ij})
$$
and
$$
\begin{aligned}
	(\delta df)_{ij}&=g^{pq}{\nabla}_{\!p}(df)_{ijq}=\frac{1}{3}g^{pq}{\nabla}_{\!p}({\nabla}_{\!i}f_{jq}+{\nabla}_{\!j}f_{jq}+{\nabla}_{\!q}f_{ij})\\
	&=\frac{1}{3}g^{pq}({\nabla}_{\!p}({\nabla}_{\!i}f_{jq}+{\nabla}_{\!j}f_{iq}+{\nabla}_{\!p}{\nabla}_{\!q}f_{ij})
	=\frac{1}{3}\big({\nabla}_{\!p}{\nabla}_{\!i}f^p_j+{\nabla}_{\!p}{\nabla}_{\!j}f^p_i+(\Delta f)_{ij}\big).
\end{aligned}
$$
From this
$$
{\nabla}_{\!p}{\nabla}_{\!i}f^p_j+{\nabla}_{\!p}{\nabla}_{\!j}f^p_i+(\Delta f)_{ij}=3(\delta df)_{ij}.
$$
Substituting this expression into \eqref{6.22}, we obtain
\begin{equation}
	jd^2=\frac{1}{3}\,d\delta+\frac{1}{2}\,\delta d+\frac{1}{6}\,d^2 j\quad\mbox{on}\ S^2.
                                     	\label{6.23}
\end{equation}

Contracting the equality \eqref{6.21} with the metric tensor $g^{ij}$, we obtain
$$
j^2d^2=\frac{2}{3}\,\delta^2+\frac{1}{6}\,\Delta\,j+\frac{1}{6}\,j\,\Delta.
$$
The operators $\Delta$ and $j$ commute. Therefore
\begin{equation}
	j^2d^2=\frac{2}{3}\,\delta^2+\frac{1}{3}\,\Delta\,j \quad\mbox{on}\ S^2.
                                   	\label{6.24}
\end{equation}

Substituting \eqref{6.23} and \eqref{6.24} into \eqref{6.20}, we have
\begin{equation}
	D^{(2)}=-\frac{1}{6}\,d\delta-\frac{1}{4}\, \delta d-\frac{1}{24}\,i\Delta\,j-\frac{1}{12}\, d^2 j
	-\frac{1}{12}\,i\delta^2.
                                 	\label{6.25}
\end{equation}
The first three terms on the right-hand side are self-adjoint operators while two last terms are adjoint to each other. Therefore this formula implies that $D^{(2)}$ is a self-adjoint operator. The coincidence of coefficients at two last terms on the right-hand side of \eqref{6.25} looks as a good fortune. The main difficulty of the proof of Conjecture \ref{Cj6.1} in the general case is just getting such coincidences.

We have thus proved the Conjecture \ref{Cj6.1} for $m=0,1,2$. Observe that we have used no specifics of the sphere, i.e., ${\S}^{n-1}$ can be replaced with an arbitrary compact Riemannian manifold in \eqref{6.18}.

As we have mentioned, the positiveness of the operator $D^{(m)}$ on the sphere follows from its self-adjointness in virtue of Theorem \ref{Th5.1}. The following fact is also of some interest: the non-negativeness of the operator $D^{(2)}$ on an arbitrary compact Riemannian manifold can be derived from \eqref{6.25}. We do not present the derivation.

\subsection{Second order Reshetnyak formula}

According to Theorem \ref{Th5.1}, the second order Reshetnyak formula
\begin{equation}
\begin{aligned}
		\|If\|^2_{H^{(2,s+1/2)}_{t+1/2}(T{\S}^{n-1})}&
		=\int\limits_0^\infty\rho^{2t+n+3}(1+\rho^2)^{s-t}
		\int\limits_{{\S}^{n-1}}\l A^{(m,2,2)}\wh f,\wh f\r(\rho\xi)\,d\xi d\rho\\
		&+\int\limits_0^\infty\rho^{2t+n+1}(1+\rho^2)^{s-t}
		\int\limits_{{\S}^{n-1}}\l A^{(m,2,1)}\wh f,\wh f\r(\rho\xi)\,d\xi d\rho\\
		&+\int\limits_0^\infty\rho^{2t+n-1}(1+\rho^2)^{s-t}
		\int\limits_{{\S}^{n-1}}\l A^{(m,2,0)}\wh f,\wh f\r(\rho\xi)\,d\xi d\rho
\end{aligned}
	                           \label{6.26}
\end{equation}
holds for any tensor field $f\in{\mathcal S}_{sol}({\R}^n;S^2{\R}^n)$. Here $s\in\R$ is arbitrary and $t>-n/2$.

We again have by \eqref{5.43}
\begin{equation}
	A^{(m,2,l)}=\frac{1}{2}\big({\tilde A}{}^{(m,2,l)}+({\tilde A}{}^{(m,2,l)})^*\big)\quad(l=0,1,2)
                                    	\label{6.27}
\end{equation}
The operators ${\tilde A}{}^{(m,2,0)},{\tilde A}{}^{(m,2,1)},{\tilde A}{}^{(m,2,2)}$ are computed by the same scheme as in the previous subsection, but all calculations are more bulky. We present the result.

First of all
\begin{equation}
	{\tilde A}{}^{(0,2,0)}={\mathbf 1},\quad {\tilde A}{}^{(1,2,0)}=(n-1){\mathbf 1}
                                   	\label{6.28}
\end{equation}
and
\begin{equation}
	{\tilde A}{}^{(m,2,0)}=\sum\limits_{p=0}^{[m/2]}\alpha_p^{(m,2,0)}i^pj^{p}\quad(m\ge2),
                                   	\label{6.29}
\end{equation}
where
\begin{equation}
\begin{aligned}
		\alpha_p^{(m,2,0)}&=\frac{\Gamma\big(\frac{n-1}{2}\big)\Gamma(m\!+\!1/2)}{\pi^{1/2}\Gamma\big(m\!+\!\frac{n-1}{2}\big)}\,
		\frac{2^{m-2p}m^2m!(m-1)!(m-2)!}{(m-2p)!(p!)^2(2m)!}\times\\
		&\times\big[
		(m-1)(m^2\!+\!mn\!-\!3m\!+\!1)^2
		+4(2m\!+\!n\!-\!3)(m^2\!+\!mn\!-\!3m\!+\!1)p\\
		&-4(m\!+\!1)(2m\!+\!n\!-\!3)(2m\!+\!n\!-\!5)p^2\big]\qquad(m\ge2).
\end{aligned}
	                                \label{6.30}
\end{equation}

Next,
\begin{equation}
	{\tilde A}{}^{(m,2,1)}=-\sum\limits_{p=0}^{[m/2]}\alpha_p^{(m,2,1)}i^pj^{p+1}d^2,
                                  	\label{6.31}
\end{equation}
where
\begin{equation}
\begin{aligned}
	\alpha_p^{(m,2,1)}&=\frac{2^{m+1}m!(m\!+\!1)!(m\!+\!2)!(m^2\!+\!mn\!-\!m\!+\!n\!-\!1)
		\Gamma\big(\frac{n-1}{2}\big)\Gamma(m\!+\!3/2)}{\pi^{1/2}(2m+2)!\Gamma\big(m\!+\!\frac{n+1}{2}\big)}\times\\
	&\times\frac{1}{2^{2p}p!(p\!+\!1)!(m\!-\!2p)!}.
\end{aligned}
	                              \label{6.32}
\end{equation}

Finally,
\begin{equation}
	{\tilde A}{}^{(m,2,2)}=\sum\limits_{p=0}^{[m/2]}\alpha_p^{(m,2,2)}i^pj^{p+2}d^4,
                                 	\label{6.33}
\end{equation}
where
\begin{equation}
	\alpha_p^{(m,2,2)}=\frac{2^m m!(m\!+\!2)!(m\!+\!4)!
		\Gamma\big(\frac{n-1}{2}\big)\Gamma(m\!+\!5/2)}{\pi^{1/2}(2m+4)!\Gamma\big(m\!+\!\frac{n+3}{2}\big)}
	\,\frac{1}{2^{2p}p!(p+2)!(m-2p)!}.
                                  	\label{6.34}
\end{equation}

As is seen from \eqref{6.29}, ${\tilde A}{}^{(m,2,0)}$ is a self-adjoint operator. Therefore $A^{(m,2,0)}={\tilde A}{}^{(m,2,0)}$ by \eqref{6.27}. Theorem \ref{Th5.1} guarantees that $A^{(m,2,0)}$ is a positive operator. The latter fact can be also derived from \eqref{6.28}--\eqref{6.30}. Indeed, all coefficients in \eqref{6.29} are positive as one can easily check by an elementary analysis of the quadratic trinomial in brackets on the right-hand side of \eqref{6.30}.

For small values of $m$, the operators ${\tilde A}{}^{(m,2,1)}$ and ${\tilde A}{}^{(m,2,2)}$ look as follows:
\begin{equation}
{\tilde A}{}^{(0,2,1)}=-\Delta,\quad {\tilde A}{}^{(0,2,2)}=\frac{1}{n^2-1}(\Delta^2+2\,\delta^2d^2);
	                                          \label{6.35}
\end{equation}
$$
{\tilde A}{}^{(1,2,1)}=-\frac{2(2n-1)}{n^2-1}\,(d\delta+2\delta d),
$$
\begin{equation}
{\tilde A}{}^{(1,2,2)}=\frac{1}{(n^2-1)(n+3)}(6\delta^2 d^2+4\delta d\delta d+d\delta d\delta+2d\delta^2 d+2\delta d^2\delta);
	                                          \label{6.36}
\end{equation}
\begin{equation}
{\tilde A}{}^{(2,2,1)}=-\frac{60(3n+1)}{(n^2-1)(n+3)}(6\,\delta d+4\,d\delta+i\delta dj+2\,d^2j),
	                                          \label{6.37}
\end{equation}
\begin{equation}
\begin{aligned}
{\tilde A}{}^{(2,2,2)}&=\frac{6}{(n^2\!-\!1)(n\!+\!3)(n\!+\!5)}
		\Big(24\,\delta^2d^2+4\,d^2\delta^2+18\,\delta d\delta d+8\,d\delta d\delta+12\,d\delta^2 d+12\,\delta d^2\delta\\
		&+2\,i\delta^2d^2j+i\delta d\delta dj+6\,\delta d^3j+6\,i\delta^3d+4\,d\delta d^2j+4\,i\delta^2d\delta+2\,d^2\delta dj+2i\delta d\delta^2\Big).
	\end{aligned}
	                                          \label{6.38}
\end{equation}

As is seen from \eqref{6.37}, ${\tilde A}{}^{(2,2,1)}$ is not a self-adjoint operator. Indeed, three first terms on the right-hand side of \eqref{6.37} are self-adjoint operators, but $d^2j$ is not self-adjoint. Thus, the symmetrization \eqref{5.43} is an essential step of our algorithm for computing $A^{(m,r,l)}$.

For the operators ${\tilde A}{}^{(m,2,2)}$, the same question can be asked as in the Conjecture \ref{Cj6.1}: is ${\tilde A}{}^{(m,2,2)}$ a self-adjoint operator? The answer is positive for $m=0,1,2$ as is seen from \eqref{6.35}, \eqref{6.36} and \eqref{6.38}. For a general $m$, the question remains open.

\end{document}